\documentclass[11pt]{amsart}

\usepackage[utf8]{inputenc}
\usepackage{amsthm}
\usepackage{amssymb}
\usepackage{amsfonts}
\usepackage{cite}
\usepackage{enumerate}
\usepackage{fullpage}
\usepackage{hyperref}
\usepackage{bbm}
\usepackage{color}
\usepackage{graphicx}
\usepackage{cleveref}

\newtheorem{theorem}{Theorem}[section]
\newtheorem{lemma}[theorem]{Lemma}
\newtheorem{prop}[theorem]{Proposition}
\newtheorem{definition}[theorem]{Definition}
\newtheorem{claim}[theorem]{Claim}

\DeclareMathOperator*{\argmax}{argmax}

\def\le{\leqslant}
\def\leq{\leqslant}
\def\ge{\geqslant}
\def\geq{\geqslant}

\def\P{\mathbb{P}}
\def\N{\mathbb{N}}
\def\eps{\varepsilon}
\usepackage[left]{lineno}
\title{Ramsey goodness of trees in random graphs}
\author{Pedro Araújo} 
\author{Luiz Moreira}
\author{Matías Pavez-Signé}

\address{Institute of Computer Science of the Czech Academy of Sciences, Pod Vodárenskou věží, 2
182 07 Praha 8, Czech Republic}
\email{pedrocampos.araujo@gmail.com}
\address{Departamento de Matemática, PUC-Rio, Rua Marquês de São Vicente, 225, Gávea, Rio de Janeiro 22451-900, Brazil}
\email{lzplm@gmail.com}
\address{Center for Mathematical Modelling (CNRS IRL-2807), Universidad de Chile, Beauchef 851, Santiago, Chile}
\email{mpavez@dim.uchile.cl}
\thanks{This research of PA and LM was partially supported by CNPq, and the research of MPS was partially supported by ANID	Doctoral scholarship ANID-PFCHA/Doctorado Nacional/2017-21171132 and the project 
UCH-1566 ``Consolidaci\'on de la internacionalizaci\'on de la investigaci\'on y postgrado de la Universidad de Chile".}

\begin{document}

\maketitle
%%%%%%%%%%%%%%%%%%%%%%%%%INTRODUCTION%%%%%%%%%%%%%%%%%%%%%%%%%
\begin{abstract} For a graph $G$, we write $G\rightarrow \big(K_{r+1},\mathcal{T}(n,D)\big)$ if every blue-red colouring of the edges of $G$ contains either a blue copy of $K_{r+1}$, or a red copy of each tree with $n$ edges and maximum degree at most $D$. In 1977, Chvátal proved that for any integers $r,n,D \ge 2$, $K_N \rightarrow \big(K_{r+1},\mathcal{T}(n,D)\big)$ if and only if $N \ge rn+1$. We prove a random analogue of Chvátal's theorem for bounded degree trees, that is, we show that for each $r,D\ge 2$ there exist constants $C,C'>0$ such that if $p \ge  C{n}^{-2/(r+2)}$ and $N \geq rn + C'/p$, then
\[G(N,p) \rightarrow \big(K_{r+1},\mathcal{T}(n,D)\big)\]
with high probability as $n\to \infty$. The proof combines a stability argument with the embedding of trees in expander graphs. Furthermore, the proof of the stability result is based on a sparse random analogue of the Erd\H{o}s--S\'os conjecture for trees with linear size and bounded maximum degree, which may be of independent interest. 
\end{abstract}

\section{Introduction}
Ever since the seminal work of Erd\H{o}s and R\'enyi~\cite{erdHos1960evolution}, the study of the binomial random graph has played a central role in combinatorics. In this paper, we study the Ramsey properties of the Erd\H{o}s--R\'enyi random graph, continuing a line of research that was initiated in the 1980s by Frankl and R\"odl~\cite{franklrodl} and by \L uczak, Ruci\'nski, and Voigt~\cite{voigt}. Let us write $G \to (H_1,H_2)$ to denote that every blue-red colouring of the edges of $G$ contains either a blue copy of $H_1$ or a red copy of~$H_2$ (if $H_1 = H_2$ then we write $G \to H$). An important early breakthrough by R\"odl and Ruci\'nski \cite{RR1,RR2} established the following threshold result for fixed non-acyclic graphs $H$:
$$\lim_{N\to\infty}\P\big( G(N,p) \rightarrow H \big) = 
\begin{cases}
1 & \text{if } {p \gg N^{-1/m_2(H)}},\\
0&\text{if } p \ll N^{-1/m_2(H)},
\end{cases}$$
where $m_2(H) = \max\big\{ \frac{e(H')-1}{v(H')-2}: H'\subseteq H \text{ with } v(H')\geq 3 \big\}$. A corresponding result for hypergraphs was obtained by Friedgut, R\"odl and Schacht~\cite{friedgut2010ramsey} and independently by Conlon and Gowers~\cite{conlongowers}, and the $1$-statement of an asymmetric version (conjectured by Kreuter and Kohayakawa~\cite{yoshi} in 1997) was recently proved by Mousset, Nenadov, and Samotij~\cite{mousset2018towards}.

Ramsey properties of random graphs involving sparse graphs have also attracted significant attention in recent years. To give just two examples, Letzter~\cite{letzter2016path} proved that if $\eps > 0$ and $pn \to \infty$, then $G\big((3/2 + \eps)n, p\big) \to P_n$ with high probability (the constant $3/2$ is best possible), and Kohayakawa, Mota and Schacht~\cite{kohayakawa2019monochromatic} proved that $\big( \frac{\log N}{N} \big)^{1/2}$ is the threshold for the event that for any two-colouring of the edges of $G(N,p)$, there exist two monochromatic trees that partition the vertex set. 

In this paper we will be interested in the problem of extending to the setting of sparse random graphs a theorem of Chv\'atal~\cite{chvatal1977tree} from 1977, which states that if $r \in \N$, and $T$ is a tree with $n$ edges, then
\begin{equation*}\label{eq:Chvatal}
K_N \to (K_{r+1},T) \qquad \Leftrightarrow \qquad N \ge rn + 1.
\end{equation*}
The necessity of the lower bound on $N$ is easy to see, and (as was first observed by Burr~\cite{burr1981ramsey}) holds in significantly greater generality. To be precise, if $H$ is a connected graph, $F$ is a graph with $\sigma(F) \le |H|$, where $\sigma(F)$ is the minimum size of a colour class in a proper $\chi(F)$-colouring of $F$, and $N < \big( \chi(F) - 1 \big) \big( |H| - 1 \big) + \sigma(F)$, then $K_N \nrightarrow (F,H).$ Indeed, to see this it suffices to consider $\chi(F)-1$ disjoint red cliques of size $|H|-1$, and one additional disjoint red clique of size $\sigma(F)-1$. A (connected) graph $H$ is said to be \emph{Ramsey $F$-good} (or just \emph{$F$-good}) if $K_N \to (F,H)$ whenever $N \ge \big( \chi(F) - 1 \big) \big( |H| - 1 \big) + \sigma(F)$. The systematic study of Ramsey goodness was initiated by Burr and Erd\H{o}s~\cite{burrerdos} in 1983. 

As far as we are aware, the problem of Ramsey goodness in random graphs was first studied only very recently, by the second author~\cite{luiz}, who considered the case in which $F$ is a clique and $H$ is a path. The main results of~\cite{luiz} identified two different thresholds for the event that $G(N,p) \to (K_{r+1},P_n)$, for different values of $N$. More precisely, it was proved there that if $p \gg n^{-2/(r + 2)}$ and $t \gg 1/p$, then $G\big( rn + t, p \big) \to \big( K_{r + 1}, P_n \big),$ while if $p \gg n^{-2/(r + 1)}$ and $t = \Omega(n)$ then $G\big( rn + t, p \big) \to \big( K_{r + 1}, P_n \big),$ in both cases with high probability as $n \to \infty$. These results are sharp in the sense that, with high probability, $G(rn + t, p) \nrightarrow  (K_{r + 1}, P_n)$ in three different settings.  First, if $p\in (0,1)$ and $t\ll 1/p$, then one can partition $V(G(N,p))=V_0\cup V_1\cup\cdots \cup V_r$ such that $|V_0|=t$ and $e(V_0,V_r)=0$. This is possible since, with high probability, every set of size $o(1/p)$ has $o(n)$ external neighbours in $G(N,p)$. Then we can colour the edges in red if and only if they have both endpoints in the same part, without creating a blue $K_{r+1}$ or any red component with more than $n$ vertices. Second, for $n^{-2/(r + 1)} \ll p \ll n^{-2/(r + 2)} $, one can show that there are values of $t\gg 1/p$ such that $G\big( rn + t, p \big) \nrightarrow \big( K_{r + 1}, P_n \big)$. Finally, if $p \ll n^{-2/(r + 1)}$ and $t = O(n)$, then, with high probability, $G(N,p)$ has $o(n)$ copies of $K_{r+1}$, whose edges can be all coloured in red without creating any red component with more than $n$ vertices, see~\cite{luiz} for the details.

Our main theorems generalise the results of~\cite{luiz} from paths to arbitrary bounded degree trees. Let us denote by $\mathcal{T}(n,D)$ the class of all trees with $n$ edges and maximum degree at most $D$. We write $G\rightarrow (K_{r+1}, \mathcal{T}(n,D))$ to denote that $G \rightarrow (K_{r+1},T)$ for every $T\in \mathcal{T}(n,D)$.

\begin{theorem}\label{stab}
For each $r,D \geq 2$, there exist $C,C' > 0$ such that the following holds. If 
$$p \geq Cn^{-2/(r+2)} \qquad \text{and} \qquad N \geq rn + C'/p,$$ 
then $G(N,p) \rightarrow \big(K_{r+1},\mathcal{T}(n,D)\big)$ with high probability as $n \to \infty$. 
\end{theorem}

As mentioned above, it follows from the results of~\cite{luiz} that the bound on $N$ is sharp up to the value of $C$, and the bound on $p$ is sharp up to a the value of $C'$. For smaller values of $p$ we obtain the following bound. 

\begin{theorem}\label{nstab}
For every $r,D\geq2$ and $\varepsilon>0$, there exists $C > 0$ such that the following holds. If 
\[p \geq  Cn^{-2/(r+1)} \qquad \text{and} \qquad N \geq rn + \varepsilon n,\]
then $G(N,p) \rightarrow \big(K_{r+1},\mathcal{T}(n,D)\big)$ with high probability as $n \to \infty$. 
\end{theorem}

In Section~\ref{sec:nstab}, we prove a stronger version of Theorem~\ref{nstab}, with a more accurate bound on $N$ and also allowing $D$ to be a function of $p$ (see Theorem~\ref{nstabv2}). We will prove Theorem~\ref{nstab} by iteratively applying a theorem due to Haxell~\cite{haxell} to find either red copies of every tree in $\mathcal{T}(n,D)$, or $r+1$ large disjoint sets with only blue edges between them. The result will then follow by a straightforward application of Janson's inequality. The proof of Theorem~\ref{stab} is significantly more challenging, and is based on a stability argument. One of the key steps is to prove that the random graph not only contains all large bounded degree trees, but is also resilient with respect to this property.

 Resilience is a measure of how much one has to perturb a graph in order to destroy a given property of it (see e.g.~\cite{botsurvey} for a discussion on resilience in the random graph) and it is a convenient way of phrasing extremal problems in general settings. For example, a classical result of Komlós, Sárközy and Szemerédi\cite{KSS1995} says that given $\delta>0$ and $n$ sufficiently large, every $n$-vertex graph $G$ with $\delta(G)\geq (1/2+\delta)n$ is universal\footnote{Given a family of graphs $\mathcal{F}$ and a graph $G$, we say that $G$ is $\mathcal{F}$-universal if $G$ contains every graph in $\mathcal{F}$ as a subgraph.} for the class of spanning trees with bounded degree. In other words, one can say that even if an adversary deletes a $(1/2-\delta)$-proportion of the edges incident at each vertex of $K_n$, the resulting graph is still universal for the class of spanning trees with bounded degree. Balogh, Csaba and Samotij~\cite{Samotij} proved that the same happens in the random graph for the class of almost spanning trees with bounded degree, provided that $p\ge C/n$ for some large constant $C$. That is, they showed that, with high probability, any subgraph of $G(n,p)$ obtained by deleting at most a $(1/2-o(1))$-proportion of the edges incident to each vertex of $G(n,p)$ is $\mathcal T(n-o(n),D)$-universal.  

One of the main features introduced in~\cite{Samotij} was an embedding technique for trees in bipartite expander graphs which works well together with the sparse regularity lemma. We combine these tools with the approach of Besomi, Stein and the third author~\cite{BPS3} to the Erd\H os--Sós Conjecture\footnote{The Erd\H os--S\'os Conjecture~\cite{Erdos64} from 1964 states that, given $k\in\mathbb N$, every graph with average degree greater than $k-1$ must contain a copy of each tree with $k$ edges. }, for bounded degree trees and dense host graphs, to obtain the following ``global" resilience
result.

\begin{theorem}\label{resilience}For every $D \geq 2$ and $\delta,\varrho\in(0,1)$, there exists $C > 0$ such that if $p\geq C/N$, then $G=G(N,p)$, with high probability, has the following property. Every subgraph $G' \subseteq G$ with $e(G')\geq \left(\varrho+\delta \right)e(G)$ is $\mathcal{T}(\varrho N,D)$-universal.\end{theorem}

Theorem~\ref{resilience} is a consequence of a stronger result in which $G(N,p)$ can be replaced by a pseudorandom graph. More precisely, we only ask that the number of edges between any pair of disjoint sets of linear size is roughly what one would expect in $G(N,p)$. This result can be viewed as an approximate random analogue of the Erd\H os--Sós conjecture for bounded degree trees of linear size. We point out that Theorem~\ref{resilience} is sharp in the following senses. The value of $p$ is best possible, up to a constant factor, since the largest connected component of $G(N,p)$ is sublinear when $p\ll 1/N$. Moreover, for an integer $r\ge 2$ and $\varrho=1/r$, the constant $\varrho$ cannot be improved. Indeed, one can partition the vertex set into $r+1$ parts, one with at most $r$ vertices and the remaining parts having the same size and thus with fewer than $N/r$ vertices. With high probability, the subgraph $G'\subseteq G(N,p)$ obtained by removing edges between parts has $(1/r-o(1))e(G(N,p))$ edges but every connected component of $G'$ has less than $N/r$ vertices.

The remainder of the paper is organised as follows. In Section~\ref{sec:outline} we give an outline of the proofs of Theorems~\ref{stab} and~\ref{resilience}. In Section~\ref{sec:expanders} we state a series of results regarding tree embeddings in expander graphs, and then we prove Theorem~\ref{nstab} in Section~\ref{sec:nstab}. In Section~\ref{sec:tools} we recall the sparse regularity lemma and some facts about the random graph. We prove Theorem~\ref{resilience} in Section~\ref{sec:resilience}, and then, putting everything together, we prove Theorem~\ref{stab} in Section~\ref{sec:stab}. %Finally, we leave some remarks and open questions to Section~\ref{sec:remarks}.

\section{Overview}\label{sec:outline}
In this section we give a rough sketch of the proofs of Theorems~\ref{stab} and~\ref{resilience}.

\subsection{The proof of Theorem~\ref{resilience}} We will use the so-called \textit{regularity method} for sparse graphs. Let $G'\subseteq G(N,p)$ be a graph with $e(G')\ge (\varrho+\delta)e(G(N,p))$. Using the sparse regularity lemma (see Section~\ref{sec:tools}) one finds a regular partition of $V(G')$ such that its corresponding reduced graph $R$ has edge density at least $\varrho+\delta/2$. To avoid confusion, we will refer to the vertices of $R$ as clusters and we set $k=|V(R)|$ for the number of clusters. By removing vertices from $R$ with less than $(\varrho+\delta/2)k/2$ neighbours, one by one, we can find an induced subgraph $R'\subseteq R$ with average degree at least $(\varrho+\delta/2)k$ and minimum degree at least $(\varrho+\delta/2)k/2$.

The lower bound on the average degree of $R'$ implies that there is a cluster $X\in V(R')$ such that $|N_{R'}(X)|\ge (\varrho+\delta/2)k$. We can partition $N_{R'}(X)$ into a matching $\mathcal{M}$ and an independent set $\mathcal Y$ so that every cluster in $\mathcal Y$ has a large neighbourhood outside $N_{R'}(X)$ (see Figure~\ref{figure:1}). We will use this structure in order to embed every tree from $\mathcal T(\varrho n,D)$. 
\begin{figure}[h!]
    \centering
    \includegraphics[scale=.56]{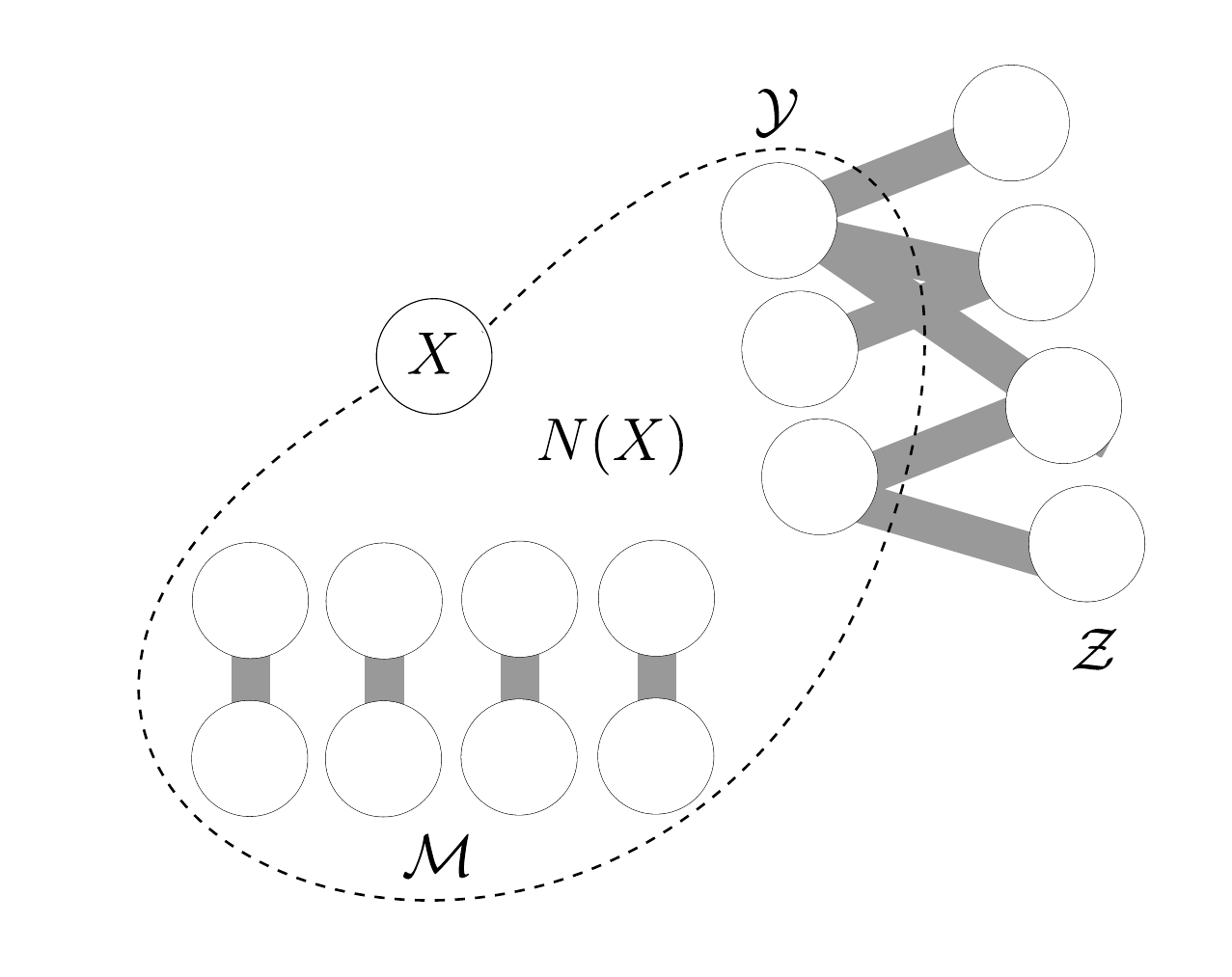}
    \label{figure:1}
    \caption{Structure in the reduced graph}
\end{figure}

The general idea is to partition a tree, embed each part into regular pairs and connect them through $X$. As an illustrative example, let us consider the case of a path $P$ with $\varrho n$ edges. We first cut $P$ into a constant number of small subpaths of odd length. We embed $P=P_1\dots  P_t$ sequentially path-by-path, in such a way that the embedding of $P$ remains connected at each step. Let  $\mathcal{H}$ be the bipartite graph induced by $\mathcal{Y}$ and $\mathcal{Z}= N_{R'}(\mathcal{Y})\setminus(X\cup N_{R'}(X))$. Starting with $P_1$, we embed the starting point of $P$ in $X$ and continue the embedding of $P_1$ into some edge either from $\mathcal M$ or $\mathcal H$. In general, the starting point of each subpath $P_i$ is embedded into $X$, and the rest of $P_i$ is embedded into some edge either from $\mathcal M$ or $\mathcal H$. Since $\mathcal{H}$ is bipartite and the number of vertices of $P_i$ is odd, the last vertex of $P_i$ can be embedded into a vertex having a large neighbourhood in $X$. This allows us to continue with the embedding of $P_{i+1}$, and so on.
 
 The proof for an arbitrary tree $T\in \mathcal{T}(\varrho n,D)$ follows the same general strategy. We first  split $T$ into a family of small rooted subtrees, and we ensure that the roots of the subtrees are at even distance from each other. The embedding of $T$ is done subtree-by-subtree following a breath first search, so that the root of each of small subtree is embedded into $X$ and the other vertices into some edge either from $\mathcal M$ or $\mathcal H$. Since $X$ is adjacent to both sides of every edge of $\mathcal M$, we can embed each subtree assigned to $\mathcal M$ in a balanced way, i.e., choosing to embed the largest bipartition class of a subtree in the cluster with the least amount of used vertices at each step. This will guarantee that almost all vertices in $\mathcal{M}$ will be used, provided that the subtrees are small enough compared to the size of the clusters. If $\mathcal M$ is large enough, then we can embed $T$ using only $\mathcal M$, but otherwise, we have to use $\mathcal H$. The main obstacle that appears while using $\mathcal H$ is that the bipartition classes of the subtrees might be unbalanced. This may be problematic because the strategy used to embed the roots in $X$ implies that the vertices of $T$ that are embedded in $\mathcal{Y}$ are all in the same bipartition class, in which case it might be impossible to use up almost all vertices in $\mathcal{Y}$, as we might run out of space in $\mathcal{Z}$. We solve this problem by assigning trees to $\mathcal{Y}$ so that we always use up more vertices in $\mathcal{Y}$ than in $\mathcal{Z}$. Therefore, if a cluster $Y\in \mathcal{Y}$ had no neighbours with spare room to embed a subtree, this would imply that we would have filled at least $2|N_{\mathcal{H}}(Y)|$ clusters of $\mathcal{H}$. The minimum degree of $R'$ is then enough to guarantee that we can go on with the aforementioned strategy.
 
\subsection{The proof of Theorem \ref{stab}} We will use a stability argument, together with Theorem 1.3, and some additional tools for embedding trees in expander graphs. Let us consider a \textit{typical} outcome of $G=G(N,p)$, where $N=rn+\Omega(1/p)$, and an arbitrary blue-red colouring of its edges with no blue copies of $K_{r+1}$ and no red copies of some tree in $\mathcal{T}(n,D)$. We divide the proof in the following steps.
\subsubsection{Rough structure of the colouring} Let $\eps,\alpha>0$ be small constants. Since the red graph $G_R$ is not $\mathcal T(n,D)$-universal, using Theorem~\ref{resilience} and the Erd\H{o}s-Simonovits stability theorem for sparse graphs (see Theorem~\ref{erdsim}) we show that the blue graph $G_B$ is close to $r$-partite. That is, there is a partition of the vertex set $V(G)=W_1\cup\dots\cup W_r$ such that $G[W_i]$ has at most $\eps pN^2$ blue edges for each $i\in [r]$.  For $i\in[r]$, since $e_B(W_i)\le \eps pn^2$ we can prove that there exists a large subset $V_i\subseteq W_i$ such that $G_R[V_i]$ is an \textit{expander graph}. We then show that $|V_i|=(1\pm o(1))n$ for each $i\in [r]$. Indeed, if $|V_i|\ge (1+\alpha)n$ for some $i\in[r]$ and $\alpha>0$, then using a theorem due to Haxell~\cite{haxell} (see Theorem~\ref{haxxel}) we deduce that $G_R[V_i]$ is $\mathcal T(n,D)$-universal, which is a contradiction with our assumption. Since no part is too large and not many vertices are removed, all the parts must have approximately the same size. Setting $V_0=V(G)\setminus (V_1\cup \dots \cup V_r)$, we obtain a partition  $V(G)=V_0\cup V_1\cup\dots\cup V_r$ such that $|V_0|\le \alpha N$ and $|V_i|=(1\pm \alpha)n$ for each $i\in [r]$.
\subsubsection{Refined structure of the colouring} In the second step we remove from each $V_i$ those vertices having a large blue neighbourhood in $V_i$ and the vertices having few neighbours in some $V_j$. Let $V(G)=V_0'\cup V_1'\cup\dots\cup V_r'$ be the resulting partition. We show that $e(V_i',V_j')=0$ for every $1\leq i<j\leq r$. From this we apply Janson's inequality to derive that any vertex $v\in V_0'$ with $\Omega(pN)$ blue neighbours in every other part would span a blue $K_{r+1}$. Moreover, we show that for all but $O(1/p)$ vertices  $v\in V'_0$ there is a unique $i\in [r]$ such that $d_B(v,V'_i)=o(pN)$ and $d_R(v,V'_i)=\Omega(pN)$, so we update the partition by setting $V_i':=V'_i\cup\{v\}$ and $V'_0=V'_0\setminus \{v\}$. Repeating this argument, we relocate vertices from $V'_0$ until only $O(1/p)$ vertices remain, and thus we end up with a partition $V=U_0\cup U_1\cup\dots \cup U_r$ such that $|U_0|=O(1/p)$, and for each $i\in[r]$ we have $\Delta(G_B[U_i])=o(pN)$ and $\delta(G_R[U_i])=\Omega(pN)$.
\subsubsection{Embedding of trees in expander graphs} Let $i^*\in [r]$ be such that $|U_{i^*}|$ is maximal. Since $N=rn+\Omega(1/p)$,  we have that $|U_{i^*}|=n+\Omega(1/p)$ and thus we must deal with the problem of embedding trees from $\mathcal{T}(n,D)$ in expander graphs of order $n+\Omega(1/p)$. This is the final aspect of the proof of Theorem~\ref{stab}.

The case of trees with at most $n/\log^4n$ leaves is covered by a theorem of Montgomery~\cite{montgomery}, which says that expander graphs are universal for the class of spanning trees with bounded degree and at most $n/\log ^4n$ leaves. For trees with at least $n/\log^4 n$ leaves, previous results in the literature do not fit in our context. Nevertheless, we may use an intermediate step in the proof of theorem of Haxell~\cite{haxell} (Theorem~\ref{haxxel}) which gives sufficient conditions to extend the partial embedding of a tree by adding a leaf at each step. To use this result we need to guarantee two conditions at each step. The first one is that the host graph has ``good"  expansion properties, and the second is that the partial embedding does not concentrate the expansion of the host graph. However, this strategy reaches the following barrier in our context. There might be two disjoint sets of sizes $\omega(1/p)$ and $n/\log^4 n$, respectively, with no edges in between. To see why this is an impediment, let $T\in \mathcal{T}(n,D)$ be a tree with at least $n/\log^4n$ leaves and let $T'\subseteq T$ be the subtree obtained by removing the leaves from $T$. Suppose that exists an embedding of $T'$ in $G_R[U_{i^*}]$. We can extend the embedding of $T'$ to an embedding of $T$ if and only if we can guarantee a certain Hall-type condition in the bipartite graph induced by the image of the parents of leaves and the set of unused vertices in $U_{i^\star}$. However, this graph might have $\omega(1/p)$ isolated vertices and we have only $O(1/p)$ ``extra" vertices.

We deal with this problem beforehand in the proof of Theorem \ref{manyleaves} in Section 3. The idea is to choose a random set $R\subseteq U_{i^*}$ of size $\Omega(n/\log^4 n)$ and then prove that there exists a realisation of $R$ such that every set $X\subseteq U_{i^*}$ of size $\Omega(1/p)$ and every set $Y\subseteq R$ of size $n/\log^4n$ have at least one edge in between. With some additional work, we can embed $T'$ in $G_R[U_{i^*}]$ so that the parents of the leaves are embedded in $R$ and then we can apply Hall's Theorem to finish the embedding.

\section{Trees in expanders}\label{sec:expanders}

For a graph $H$ and a subset $X\subseteq V(H)$, we denote by $\Gamma(X)=\bigcup_{x\in X}N(x)$ the set of neighbours of $X$ and write $N(X)=\Gamma(X)\setminus X$ for the external neighbourhood of $X$. In this section, we study the family of graphs called \textit{expanders} in which subsets of vertices have a large external neighbourhood. The notion of expander graphs has a plentiful number of applications in combinatorics and it is particularly useful for embedding trees. Indeed, Friedman and Pippenger~\cite{pip} proved that given integers $m$ and $D$, if a graph $H$ satisfies
\[|\Gamma(X)|\geq (D+1)|X|\text{ for all }X\subseteq V(H)\text{ with }1\le|X|\leq 2m, \]
then $H$ contains all trees with $m$ vertices and maximum degree $D$. A limitation of this result is that it only works for trees of size at most $|V(H)|/(2D+2)$. In a successful attempt to overcome this, Haxell \cite{haxell} considered a different notion of expansion to prove the following theorem.
\begin{theorem}
\label{haxxel}
Let $D,m,t \in \mathbb{N}$ and let $H$ be a graph with the following
properties:
\begin{enumerate}
    \item$|N(X)|\geq D|X|+1$, for all $X\subseteq V(H)$
    with $1\leq |X| \leq m$.
    \item\label{hax:2} $|N(X)|\geq t+D|X|+1$, for all $X\subseteq V(H)$
    with $m+1\leq |X| \leq 2m.$
\end{enumerate}
Then $H$ contains a copy of every tree $T$ with $t$ vertices and maximum degree at most $D$. Furthermore, given $v\in V(H)$ and $u\in V(T)$,
there exists an embedding of $T$ mapping $u$ to $v$.
\end{theorem}
A different and convenient way of phrasing property~(\ref{hax:2}) of Theorem \ref{haxxel} is the following. Let $H$ be a graph such that every pair of disjoint sets $X,Y\subseteq V(H)$, with $|X|=m_1$ and $|Y|=m_2$, satisfies $e(X,Y)>0$. Then for every $Z\subseteq V(H)$, with $m_1\le |Z|\leq 2m_1$, there are at most $m_2-1$ vertices in the non-neighbourhood of $Z$. By discounting the non-neighbours of $Z$ and the vertices in $Z$, we get 
\begin{equation}
\label{expanders}
|N(Z)| \geq |V(H)| - |Z| - m_2 +1.
\end{equation}
Therefore, when $|V(H)|-m_2\geq t+2(D+1)m_1$ we recover property~(\ref{hax:2}). The main result of this section considers the case where $m_1$ and $m_2$ have different orders of magnitude, which leads us to the following definition.
\begin{definition}
Let $D, m_1,m_2$ be integers. We say that a graph $H$ is an $(m_1,m_2,D)$-expander if 
\begin{enumerate}[(i)]
    \item\label{strong} $|N(X)| \geq D|X|+1$ for all $X\subseteq V(H)$ with $1\le|X|\leq m_1$, and
    
    \item\label{weak} $e(X,Y)>0$ for all disjoint sets $X,Y\subseteq V(H)$ with $|X|=m_1$ and $|Y|=m_2$.
\end{enumerate}
Moreover, if only property~\eqref{weak} holds, then we say that $H$ is a weak $(m_1,m_2)$-expander. We will often omit $D$ when it is clear from context.
\end{definition}
As is usual with tree embedding problems, we deal separately with trees with many or few leaves. Using the Absorption Method, Montgomery \cite{montgomery} proved the following result.
\begin{theorem}
\label{mont}
  Let $n$ be sufficiently large, let $D$ be a positive integer, and set $d=D\log^4n/20$. If $H$ is a $(n/2d,n/2d,d)$-expander on $n$ vertices, then $H$ contains every spanning tree with maximum degree bounded by $D$ and at most $n/d$ leaves.
\end{theorem}
We remark that, although Theorem~\ref{mont} is not stated explicitly in~\cite{montgomery}, it follows directly from Montgomery's proof (see~\cite[Section~4.2]{montgomery}), where it is only used that $G(n,p)$ is an expander as in Theorem~\ref{mont}. The main result of this section deals with the case of (non-spanning) trees with many leaves.
\begin{theorem}
  \label{manyleaves}
Let $m_1,m_2,n, D$ be positive integers such that $6m_1\log n  < m_2$ and $16Dm_2 \leq n$, and assume that $n$ is sufficiently large. Let $H$ be a graph on $n$ vertices such that $H$ is 
\begin{enumerate}[(1)]
  \item\label{exp:2}a weak $(m_1, n/32D)$-expander, and
  \item\label{exp:3}a weak $(m_2,m_2)$-expander.
\end{enumerate}
Then $H$ contains every tree $T\in\mathcal T(n-m_1,D)$ with at least $24Dm_2$ leaves.
\end{theorem}
We remark that the proof of Theorem 3.1 in \cite{haxell} relies on a clever inductive argument in order to embed all vertices of the tree but the leaves, and then uses a Hall-type result to finish the embedding. However, the hypothesis of Theorem \ref{manyleaves} do not enable a straightforward modification of this proof for the following reason. Given a tree $T$, let $L\subseteq V(T)$ be the set of leaves of $T$ and let $P=N(L)$ be their parents. Note that if $T\in \mathcal{T}(n-m_1,D)$ is a tree with $|L|= \Omega (m_2)$ leaves, then we also have $|P|=  \Omega(m_2)$. Suppose that we have a partial embedding of $T-L$ which we want to extend to $T$. By the hypothesis of Theorem \ref{manyleaves}, it might be that the image of $P$ has $m_2-1$ non-neighbours in the leftover vertices, in which case is impossible to extend the embedding of $T-L$ since $m_1<m_2$. %Therefore, we can only hope to embed trees with $n- O(m_2)$ with this approach.

We address this obstacle by finding a set $W\subseteq V(H)$ with $\Theta(m_2)$ vertices such that every subset $X\subseteq W$ with $|X|=m_2$ has less than $m_1$ non-neighbours in $H$. We then manage to find an embedding $\varphi:V(T- L)\to V(H)$ such that $\varphi(P)\subseteq W$, in which case we would have that
\[|N(X)\setminus \varphi(V(T-L))|\geq n-|T-L| - m_1+1> |L|\]
for every $X\subseteq \varphi(P)$ with $|X|\geq m_2$. Nevertheless, is also necessary to guarantee that small subsets of $\varphi(P)$ have enough neighbours in the set of unused vertices. Motivated by this, we need to define a `good' embedding of a tree. Basically, we say that an embedding of a tree $T$ is \textit{good} if the set of used vertices satisfies a certain Hall-type condition, for small sets, to the set of unused vertices. 
\begin{definition}
Let $m$ be a positive integer, let $T$ be a tree with maximum degree at most $D$, and let $H$ be a bipartite graph with parts $V_1$ and $V_2$. We say that an embedding $\varphi:V(T)\to V(H)$ is $m$-good in $H$ if for every $i\in \{1,2\}$ and $X\subseteq V_i$, with $1\le|X|\leq m$, we have
\[ |N_H(X)\setminus \varphi(V(T))| \geq \sum_{v\in \varphi^{-1}(X)}\big(D - d_T(v)\big) + D|X\setminus \varphi(V(T))|. \]
\end{definition}
In the previous definition we considered $H$ as being bipartite for technical reasons. More specifically, as we want to embed the set of parents of leaves into a set $W$, we have to alternate the embedding of  $T$ between $W$ and $V(H)\setminus W$ and thus it is easier to consider $H$ as being a bipartite graph.  The next lemma gives sufficient conditions to extend good embeddings, and it was proved in~\cite{Samotij} as the induction step\footnote{Under the  hypothesis Theorem~7 from~\cite{Samotij}, the authors state that good embeddings can be extended as ``Property 2" in page 6 from~\cite{Samotij}. Moreover, the only place where they use the size of neighbours of sets with more than $m$ vertices is in the proof of Claim~8. One can check that~\eqref{mew} is enough to get the same proof.} in the proof of a bipartite analogue of Theorem~\ref{haxxel} (see Theorem \ref{magmar}).
\begin{lemma}
  \label{goodhaxxel}
  Let $m,n,D$ be positive integers, let $T$ be a tree with maximum degree at most $D$, and let $H$ be a bipartite graph with parts $V_1$ and $V_2$. Suppose that there exists an $m$-good embedding $\varphi:V(T)\to V(H)$, and that for $i\in\{1,2\}$ and any subset $X\subseteq V_i$, with $m\leq|X|\leq 2m$, we have
 \begin{equation}
 \label{mew} |N_H(X)\setminus \varphi(V(T))| \geq 2Dm + 2.\end{equation}
Then for every vertex $v\in T$, with $d_T(v)<D$, there exists an $m$-good embedding of the tree obtained by adding to $T$ a leaf adjacent to $v$.
\end{lemma}
We will be able to use Lemma~\ref{goodhaxxel} in graphs satisfying the following notion of `bipartite expansion'.
\begin{definition}
   Let $D\geq 2$ and let $H$ be a bipartite graph with parts $V_1$ and $V_2$ such that $|V_1|\leq |V_2|$. Let $m$ be a positive integer with $m<|V_1|$. We say that $H$ is a bipartite $(m,D)$-expander if the following two properties hold.
   \begin{enumerate}
     \item\label{bipexp:1} For $i\in \{1,2\}$, every set $X\subseteq V_i$, with $1\le |X|\le m$, satisfies $|N_H(X)|\ge D|X|$.
     \item\label{bipexp:2} For every pair of sets $X_1\subseteq V_1$ and $X_2\subseteq V_2$, each of size at least $m$, we have $e(X_1,X_2)>0$. 
   \end{enumerate}\end{definition}
Note that property~\eqref{bipexp:2} implies that for every subset $X\subseteq V_i$, with $|X|\geq m$, we have
 \[|N(X)|\geq |V_{3-1}|-m+1.\] 
This will guarantee that \eqref{mew} holds for the embedding of any tree with small enough bipartition classes. Now we can state one of the main results that we need for the proof of Theorem~\ref{manyleaves}.
\begin{lemma}
\label{choosingwhere}
Let $m,D$ be positive integers and let $T$ be a tree with maximum degree at most $D$. Let $U_1\cup U_2$ be any partition of one the bipartition classes of $T$ and let $U_3$ be the other bipartition class. Let $H$ be a graph on $n$ vertices and let $V_1,V_2,V_3\subseteq V(H)$ be disjoint sets such that $|V_i|\geq |U_i|+3Dm$ for $i\in\{1,2,3\}$. If $H[V_1, V_3]$, $H[V_2, V_3]$ and $H[V_1\cup V_2,V_3]$ are bipartite $(m,D)$-expanders, then there exists an $m$-good  embedding $\varphi:V(T)\to V(H)$ such that  $\varphi(U_i)\subseteq V_i$  for $i\in\{1,2,3\}$.
\end{lemma}

  The strategy of the proof of Lemma~\ref{choosingwhere} is to iteratively apply Lemma~\ref{goodhaxxel} in order to extend a partial embedding of the tree by adding a leaf at each step. Since we will alternate between vertices of $V_1,V_2$ and $V_3$, we will need to keep track that the embeddings are $m$-good in the graphs $H[V_1,V_3]$, $H[V_2,V_3]$ and $H[V_1\cup V_2, V_3]$, respectively. This will guarantee that, at any stage of the embedding, small subsets of $V_1\cup V_2$ have enough neighbours in the unused vertices of $V_3$, and that small subsets of $V_3$ have enough neighbours in the unused vertices of both $V_1$ and $V_2$. 
  
  In the context of Lemma~\ref{choosingwhere}, for a subtree $S\subseteq T$ we say that $\varphi:V(S)\to V(H)$ is $m$-great if
\begin{enumerate}[(I)]
\item\label{great:1} $U_i\cap V(S)$ is mapped to $V_i$, for $i\in\{1,2,3\}$, and
\item\label{great:2}    $\varphi$ is $m$-good in $H[V_1\cup V_2,V_3]$ and in $H[V_i,V_3]$, for $i\in\{1,2\}$.
\end{enumerate}
\begin{proof}[Proof of Lemma~\ref{choosingwhere}] We start by showing that there exists an $m$-great embedding of any single vertex subtree $S\subseteq T$ 
 \begin{claim}\label{claim:basecase}
  Let $S\subseteq T$ be a single vertex subtree. If $\varphi:V(S)\to V(H)$ is an embedding which satisfies property~\eqref{great:1}, then $\varphi$ is $m$-great.
 \end{claim}
 \begin{proof}[Proof of Claim~\ref{claim:basecase}]
 We will only prove that $\varphi$ is $m$-good in $H[V_1,V_3]$, as the other cases are completely analogous. Since $H[V_1,V_3]$ is a bipartite $(m,D)$-expander, then for $X\subseteq V_1$, with $m\le|X|\leq 2m$, we have 
\[|(N(X)\cap V_3)\setminus \varphi(V(S))|\geq |V_3|-|S|-m+1,\]
\noindent which is larger than the required lower bound in the definition of $m$-goodness. Since the same bound holds if $X\subseteq V_3$, it follows that $\varphi$ is $m$-good in $H[V_1,V_3]$.
 \end{proof}
Now that we have proved the base case, we will prove that any $m$-great embedding of a subtree $S\subset T$ can be extended by adding a leaf. Let $s\in V(S)$ and $v\in V(T-S)$  such that $sv\in E(T)$. Assume we have an $m$-great embedding $\varphi:V(S)\to V(H)$ and we want to add $v$. We deal separately with the cases when $v \in U_3$ or $v \in U_1\cup U_2$. 

Suppose that $v\in U_3$. Since $H[V_1\cup V_2, V_3]$ is a $(m,D)$-expander, then for $X\subseteq V_1\cup V_2$ (and analogously for  $X\subseteq V_3$), with $m\leq |X|\leq 2m$, we have that
  \begin{equation} \label{psyduck}
     |(N(X)\cap V_3) \setminus \varphi(V(S))|\ge |V_3|-m+1-|U_3|\geq 3Dm-m+1 \geq 2Dm+2. 
  \end{equation}
Thus, by Lemma~\ref{goodhaxxel}, there exists an $m$-good embedding $\varphi':V(S+sv)\to V(H[V_1\cup V_2,V_3])$. We argue now that $\varphi'$ is $m$-good in $H[V_i,V_3]$, for $i\in\{1,2\}$. Indeed, given $X\subseteq V_i$ for some $i\in\{1,2\}$, we already know that $|(N(X)\cap V_3)\setminus \varphi'(V(S))|\ge 2Dm+2$  since $\varphi'$ is $m$-good in $H[V_1\cup V_2,V_3]$. For $X\subseteq V_3$ there is nothing to prove, since $\varphi$ was $m$-great and we did not use any additional vertices from either $V_1$ or $V_2$. 

  The case when $v\in U_1$ (resp. $v\in U_2$) is analogous, but we apply Lemma~\ref{goodhaxxel} to $\varphi$ in the bipartite graph $H[V_1,V_3]$ (resp. $H[V_2,V_3]$), together with the same calculation as in~\eqref{psyduck}, to get an $m$-good embedding $\varphi'$. Note that $\varphi'(v)\in V_1$ (resp. $\varphi'(v)\in V_2$). This guarantees that $\varphi'$ is $m$-good in $H[V_1,V_3]$ and $H[V_2,V_3]$. Moreover, for $H[V_1\cup V_2,V_3]$ we only need to guarantee the neighbourhood expansion for $X\subseteq V_3$ with $m\le |X|\leq 2m$. Note that since $\varphi'$ is $m$-good in $H[V_i,V_3]$ for $i\in\{1,2\}$ we have
  \[ |(N(X)\cap (V_1\cup V_2))\setminus \varphi'(V(S))|\geq |(N(X)\cap V_1)\setminus \varphi'(V(S))|\ge 2Dm+2,\]
 and thus $\varphi'$ is $m$-good in $H[V_1\cup V_2,V_3]$.
  \end{proof}
The last ingredient that we need for Theorem~\ref{manyleaves} is a well-known generalisation of Hall's theorem.
\begin{lemma}\label{lem:Hall}Let $G$ be a bipartite graph with parts $A=\{a_1,\dots,a_\ell\}$ and $B$. Let $(d_i)_{i \in [\ell]}$ be a sequence of non-negative integers and let $(S_i)_{i\in [\ell]}$ be a collection of stars such that $S_i$ has $d_i$ leaves for each $i\in [\ell]$. Then $G$ contains an embedding of $(S_i)_{i\in[\ell]}$, with the centres mapped into $A$, if and only if
	\begin{equation}
		|N(X)|\geq \sum_{x\in X}d_{x}\text{ for all }X\subseteq A.
	\end{equation}
\end{lemma}
\begin{proof}[Proof of Theorem~\ref{manyleaves}]
 Let $L$ be a set of $12Dm_2$ leaves of $T$ in the same bipartition class and let $U_1$ be the set of parents of $L$ in $T$. Note that $ 12m_2\leq |U_1|\leq 12Dm_2$. We choose, uniformly at random, a set $W\subseteq V$ with $r=|U_1|+4Dm_2$ vertices, and note that $r \leq 16Dm_2\leq n$ . For each set $X\subseteq V(H)$ with $m_1$ vertices, let $Z_X=\{y\in W\setminus X: d(y,X)=0 \}$ . Since $H$ is a weak $(m_1,n/32D)$-expander, then  
  \[\mathbb{E}|Z_X|\leq \dfrac{r}{n}\cdot \dfrac{n}{32D} \leq  \dfrac{m_2}{2}.\]
  By standard tail bounds for the hypergeometric distribution (see Theorem 2.10 in \cite{rucinskibook}), we have 
\[ \mathbb{P}(|Z_X|\geq m_2)\leq \exp\left(- \dfrac{m_2}{6}\right).\]
Denoting by $Z$ the number of sets $X\subseteq V(H)$ of size $m_1$ such that $|Z_X| \geq m_2$, we have 
  \[ \mathbb{E}[Z] \leq n^{m_1}\exp(-m_2/6) <1,\]
since $6m_1\log n<m_2$. This implies that there is a realisation of $W$, denoted by $W_1$, such that every subset $X\subseteq V(H)$ of size $m_1$ has less than $m_2$ non-neighbours in $W_1$. Set $T'=T-L$ and suppose that one of the bipartition classes of $T'$ is $U_1\cup U_2$ and the other is $U_3$. We take two disjoint sets $W_2, W_3\subseteq V(H)\setminus W_1$ such that $|W_i|= |U_i|+4Dm_2$ for $i\in\{2,3\}$, which is possible since in this case we have
\[|W_1|+|W_2|+|W_3| =   |T|-|L|+12Dm_2 \le n.\]
\begin{claim}\label{claim:bip expanders}
 For each $i\in\{1,2,3\}$ there exists $V_i\subseteq W_i$, with $|W_i\setminus V_i|\leq 2m_2$, such that the graphs $H[V_1\cup V_2,V_3]$, $H[V_1,V_3]$ and $H[V_2,V_3]$ are bipartite $(m_2,D)$-expanders.
\end{claim}
\begin{proof}[Proof of Claim~\ref{claim:bip expanders}]
Since $H$ is a weak $(m_2,m_2)$-expander,  the second property of the bipartite expansion is already satisfied for all the three bipartite graphs. We will find the sets $V_i$'s iteratively. We set $X_i=\emptyset$ and $V_i:=W_i$ for $i\in\{1,2,3\}$. While there exists
\begin{itemize}
\item $X\subseteq V_3$ with $|X|\leq m_2$ and $|N(X)\cap V_i|< D|X|$ for some $i\in\{1,2\}$, 
we set $X_i:=X_i\cup X$ and $V_3:=V_3\setminus X$, and
\item $X\subseteq V_1\cup V_2$ with $|X|\leq m_2$ and $|N(X)\cap V_3|< D|X|$, we
set $X_3:=X_3\cup X$ and $V_i:=V_i\setminus X$ for $i\in\{1,2\}.$
\end{itemize}
First, we show that at each step we have $|X_i|\leq m_2$ for $i\in\{1,2,3\}$, and that
\[|N(X_3)\cap W_3| < D|X| \quad \text{ and } \quad |N(X_i)\cap W_i|< D|X|\]
for $i\in\{1,2\}$. Indeed, if this is satisfied for some $X_1,X_2,X_3$ and there exists $X\subseteq V_1\cup V_2$ (or analogously for $X\subseteq V_3)$ with $|N(X)\cap W_3|< D|X|$, then clearly we have that
\[|N(X_3\cup X)\cap V_3| \le |N(X_3)\cap W_3| + |N(X)\cap W_3| 
                    < D|X_3| + D|X| = D|X_3\cup X|.\]
If we had that $|X|\geq m_2$, then by the weak $(m_2,m_2)$-expansion of $H$, $X$ would have less than $m_2$ non-neighbours in $V_3$ and therefore we would have that
\[|N(X)\cap V_3|\geq |V_3|-m_2 \geq 2Dm_2\geq D|X|+1,\]
which contradicts the choice of $X$. This finishes the proof since $|X_1\cup X_2|,|X_3|\leq 2m_2$. 
\end{proof}
Let $V_i\subseteq W_i$ be the sets given by Claim~\ref{claim:bip expanders} for $i\in\{1,2,3\}$ so that $H[V_1\cup V_2,V_3]$, $H[V_1,V_3]$ and $H[V_2,V_3]$ are bipartite $(m_2,D)$-expanders. Observe that 
\[|V_i|\geq |U_i|+4Dm_2-2m_2 \geq |U_i|+3Dm_2,\]
for $i\in\{1,2,3\}$ which, by~\ref{choosingwhere}, implies that we can find an $m_2$-good embedding $\varphi':V(T')\to V(H)$ such that $\varphi'(U_i)\subseteq V_i$ for $i\in\{1,2,3\}$. To finish the embedding of $L$, we will use Lemma~\ref{lem:Hall} in the bipartite graph $H[\varphi'(U_1),V(H)\setminus \varphi'(V(T'))]$.

Note that the condition of Lemma~\ref{lem:Hall} is satisfied for every subset $X\subseteq \varphi'(U_1)$ with $|X|\leq m_2$ as $\varphi'$ is $m_2$-good and $\Delta(T)\le D$. Moreover, since $\varphi'(U_1)\subseteq W_1$ and by the choice of $W_1$, every subset $X\subseteq \varphi'(U_1)$, with $|X|\geq m_2$, has less than $m_1$ non-neighbours and therefore
\[|N(X)\cap V(H)\setminus \varphi'(V(T')) |\geq |V(H)\setminus \varphi'(V(T'))|- m_1 \geq |L|, \]
as $|T'|=|T|-|L|=n-m_1-|L|$. Then Lemma~\ref{lem:Hall} implies we can finish the embedding of $L$ and thus finish the proof.

\end{proof}

\section{Proof of Theorem \ref{nstab}}\label{sec:nstab}

The proof of Theorem \ref{nstab} follows by applying Proposition \ref{indhaxxel} $r+1$ times. For an appropriate choice of $m_1$ and $m_2$ there will be two possibilities. If the red graph is a weak $(m_1,m_2)$-expander, then,  using Theorem \ref{haxxel}, we show that it is $\mathcal{T}(n,D)$-universal. Otherwise it will contain two disjoint sets of size $m_1$ and $m_2$, respectively, with all edges in between coloured in blue. We repeat this argument $r$ times in the induced graph on the set with $m_2$ vertices. 
At the end of this process, if the red graph is not $\mathcal{T}(n,D)$-universal, then we get $r+1$ disjoint sets, each of size $m_1$, with all the edges in between coloured in blue. This reasoning is made precise in the proof of the following lemma.

\begin{lemma}
\label{weaklyclique}
  Let $n,m,r,D=D(n)$ be positive integers and let $H$ be a graph on $N=rn+10Drm$ vertices. Then one of the following holds:
  \begin{enumerate}
      \item $H$ is $\mathcal{T}(n,D)$-universal.
      \item There are disjoint sets $U_1,\dots,U_{r+1}\subseteq V(H)$, each of size $m$, such that $e(V_i,V_j)=0$ for $1\leq i<j\leq r+1$.
  \end{enumerate}
\end{lemma}

Before proving Lemma \ref{weaklyclique} we need to show that weak expander graphs contains an almost spanning expander.

\begin{prop}
  \label{indhaxxel}
Let $D,m_1, m_2$ be integers and let $H=(V,E)$ be a graph with $|V|\geq m_2+(2D+2)m_1 $. If $H$ is a weak $(m_1,m_2)$-expander, then there exists a set $V'\subseteq V$, with $|V\setminus V'|\leq m_1$, such that $H[V']$ is a $(m_1,m_2,D)$-expander.
\end{prop}
\begin{proof}
  Take a maximal set $Z\subseteq V$ with $1\le |Z|\le m_1$ and $|N(Z)|< D|Z|$, and set $V'=V\setminus Z$. We will prove that for any $X\subseteq V'$ with $|N(X)\cap V'|< D|X|$ we have that $|X|> m_1$, which shows that $H[V']$ is a $(m_1,m_2,D)$-expander. Observe that for such $X$ we have
  \[|N(X\cup Z)|< D|X\cup Z|.\]
  By the maximality of $Z$, we conclude that $|X\cup Z|> m_1$. Since $H$ is a weak $(m_1,m_2)$-expander, there are less than $m_2$ non-neighbours of $X\cup Z$ in $V'\setminus (X\cup Z)$ and then
  \begin{equation*}
    \begin{split}
      D|X|  > |N(X)\cap V'| 
           & \geq |N(X\cup Z)\cap V'| - |N(Z)\cap V'|\\
           &> |V'|- m_2 +1 - |X|-D|Z|  \\
           &= |V|-m_2+1-(D+1)m_1 -|X|.
    \end{split}
  \end{equation*}
Using that $|V|\ge m_2+(2D+2)m_1$ we get $D|X|>(D+1)m_1 - |X|$ which implies $|X|>m_1$ and finishes the proof.
\end{proof}
Now we move to the proof of Lemma \ref{weaklyclique}.

\begin{proof}[Proof of Lemma \ref{weaklyclique}]
   We assume that $H$ is not $\mathcal{T}(n,D)$-universal and set $V_0=V(H)$. We will prove that for $s\in[r]$ there exist disjoint sets $U_s$ and $V_s$  with
   \[|U_s|=m \quad \text{and} \quad |V_s|= (r-s)n + (r-s+1)5Dm,\]
    \noindent such that $e(U_s,V_s)=0$ and $U_s,V_s\subseteq V_{s-1}$. Indeed, if this is true, we set $U_{r+1}=V_r$ and get that $e(U_i,U_j)=0$ for every $1\leq i<j\leq r+1$, which is what we want to prove. We remark that throughtout this section, $D$ does not need to be a constant.

Suppose we have found sets $V_0, U_1,V_1,\cdots, U_s,V_s$ as above for some $s\in[r]$, or just $V_0$ for $s=0$, and let us show how to find $U_{s+1}$ and $V_{s+1}$. Let $m_s= (r-s-1)n +(r-s)5Dm$ and suppose that $H[V_s]$ is not a weak $(m,m_s)$-expander. Then there are disjoint sets $U_{s+1},V_{s+1}\subseteq V_s$ of size $m$ and $m_s$, respectively, such that $e(U_{s+1},V_{s+1})=0$. Therefore, we only need to prove that $H[V_s]$ is not a weak $(m,m_s)$-expander. 

Now, we show that if $H[V_s]$ were a weak $(m,m_s)$-expander, then it would be $\mathcal{T}(n,D)$-universal, which we assumed not to be true. To prove that, we first note that 
  \[  |V_s|-m_s = n+5Dm. \]
Since $|V_s|\geq (D+2)m+m_s$,  there exists a subset $V'_s\subseteq V_s$ such that $|V_s\setminus V'_s|\leq m$ and $H[V'_s]$ is $(m,m_s,D)$-expander. As reasoned in~\eqref{expanders}, for a set $X\subseteq V_{s}'$, with $m\leq |X| \leq 2m$, the $(m,m_s,D)$-expansion implies that
\begin{equation*}
 \begin{split}
     |N(X)\cap V'_s|\geq |V'_s| - m_2-|X|+1
                    &\geq |V_s| - m-m_2-2m+1\\
                    &\geq n+5Dm-3m+1\\
                    &\geq n + D|X|+1.
 \end{split}
\end{equation*}

\noindent The above inequality and the first property of the $(m,m_s,D)$-expansion imply,  by Theorem \ref{haxxel}, that  $H[V'_i]$ is $\mathcal{T}(n,D)$-universal.   
\end{proof}

Lemma \ref{weaklyclique} reduces the proof of Theorem \ref{nstab} to finding the minimum value $m$ such that every collection of $r+1$ disjoint $m$-sets, with high probability, span a copy of $K_{r+1}$ in $G(N,p)$ with one vertex in each $m$-set. Such a copy of $K_{r+1}$ will be called a \textit{canonical copy}. To do this we have the following lemma, whose proof is a standard application of Janson's inequality and therefore we omit it.
\begin{lemma}
      \label{Janson}
    Let $r\ge 2$ and let $G = G(N,p)$, with $p\gg N^{-2/(r+1)}$. Fix a disjoint collection $V_1,\dots,V_{r+1}
    \subseteq V(G)$, with $|V_i|=m_i$ for $i\in[r+1]$. Then the probability that $V_1,\cdots,V_{r+1}$ spans a canonical copy of $K_{r+1}$ is at least 
    \[ 1-\exp\left(-\Omega\left(p^{\binom{r+1}{2}} \prod_{i=1}^{r+1}m_i \right) \right).\]
    In particular, there exists a constant $C>0$ such that if an integer $m$ satisfies
    \begin{equation}
     \label{jansoncond}   
     m^{r+1}p^{\binom{r+1}{2}}\geq C\log{\binom{N}{m}}, 
     \end{equation}
then with high probability there exists a canonical copy of $K_{r+1}$ in every collection of $r+1$ disjoint $m$-sets.
\end{lemma}
Now we may state a stronger version of Theorem \ref{nstab}, with $t=O(m)$ and $m$ satisfying \eqref{jansoncond}. Note that when $t=\Omega(N)$, condition~\eqref{jansoncond} is equivalent to say that $p\geq CN^{-2/(r+1)}$, for some $C>0$.

\begin{theorem}
\label{nstabv2}
For every $r,D=D(n)\geq2$ and for every $p=p(n)$ and $m$ satisfying \eqref{jansoncond}, if
\[N \geq rn + 10Drm,\]
then $G(N,p) \rightarrow \big(K_{r+1},\mathcal{T}(n,D)\big)$ with high probability.
\end{theorem}
\begin{proof}
Let $G=G(N,p)$, where $N=rn+ 10Drm$, and consider the event in which every collection of $r+1$ disjoint sets of size $m$ span a canonical copy of $K_{r+1}$. By Lemma \ref{Janson} and the hypothesis on $m$, this happens with high probability. Let  $G_R$, $G_B\subseteq G$ be the red and blue graphs in a given edge colouring of $G$. By  Lemma \ref{weaklyclique}, if $G_R$ is not $\mathcal{T}(n,D)$-universal, then there are disjoint sets $U_1,\dots,U_{r+1}$ of size $m$ such that $e_R(U_i,U_j)=0$ for all $1\leq i<j\leq r+1$. In other words, all the edges in between these sets are coloured blue, which spans a blue copy of $K_{r+1}$ by the choice of $m$.
\end{proof}

%%%%%%%%%%%%%%%%%%%%%%%%%PRELIMINARES%%%%%%%%%%%%%%%%%%%%%%%%%
\section{Regularity and facts about the random graph}\label{sec:tools}
In this section we state some tools needed for the proof of Theorem~\ref{stab} and Theorem~\ref{resilience}. 

\subsection{The sparse random Erd\H{o}s--Simonovits stability theorem}  The following result is one of a series of random analogues of extremal results proved, independently, by Conlon and Gowers~\cite{conlongowers} and by Schacht~\cite{schacht}.

\begin{theorem}
\label{erdsim}
For every $r\geq 2$ and $\varepsilon>0$, there are positive numbers
$C'$ and $\delta$ such that for $p\geq C'N^{-2/(r+2)}$ 
the following holds. With high probability, every $K_{r+1}$-free subgraph $G$ of $G(N,p)$ with% at least
\[ e(G)\ge \left( 1 - \dfrac{1}{r} - \delta \right)p\binom{N}{2}\]
\noindent can be made $r$-partite by removing at most $\varepsilon pN^2$ edges.
\end{theorem}

    \subsection{Sparse regularity}The proof of Theorem~\ref{resilience} relies on a sparse version of the Szemer\'edi's Regularity lemma. In order to state this result, we need  some basic definitions.
\begin{definition} Let $\eta,p\in (0,1)$. We say that an $n$-vertex graph $G$ is $(\eta,p)$\textit{-uniform}, if all disjoint sets $A,B\subseteq V(G)$ with $|A|,|B|\geq \eta n$ satisfy
  \begin{equation}\label{uniform:1}
  (1-\eta)p|A||B| \leq e_G(A,B)\leq (1+\eta)p|A||B|\end{equation}
\noindent and
  \begin{equation}\label{uniform:2}(1-\eta)p\dbinom{|A|}{2} \leq e_G(A)\leq (1+\eta)p\dbinom{|A|}{2}.\end{equation}
\noindent Furthermore, we say that $G$ is $(\eta,p)$\textit{-upper-uniform} if (possibly) only the upper bounds in~\eqref{uniform:1} and~\eqref{uniform:2} hold for all $A,B\subseteq V(G)$ as above.

\end{definition}
    Let $G$ be a graph and let $p\in(0,1)$. Given two disjoint sets $A,B\subseteq V(G)$, we define the $p$-density of the pair $(A,B)$ by
    \begin{equation*}d_p(A,B)=\frac{e(A,B)}{p|A||B|}.\end{equation*}
Given $\varepsilon >0$, we say that the pair $(A,B)$ is $(\varepsilon ,p)$-regular if for all $A'\subseteq A$ and $B'\subseteq B$, with $|A'|\ge \varepsilon|A|$ and $|B'|\ge \varepsilon|B|$, we have
    \[|d_p(A',B')-d_p(A,B)|\leq \varepsilon.\]
\noindent Now we state some standard results regarding properties of regular pairs (we refer to the survey~\cite{gerke_steger_2005} for the proofs).

\begin{lemma}\label{lem:regularpairs}
      Given $0<\varepsilon<\alpha$, let $G$ be a graph and let $A,B\subseteq V(G)$ be disjoint sets such that $(A,B)$ is $(\varepsilon,p)$-regular with $d_p(A,B)=d>0$. Then the following are true.
      \begin{enumerate}
      \item For any $A'\subseteq A$ with $|A'|\geq \alpha |A|$ and $B'\subseteq B$ with $|B'|\geq \alpha |B|$, the pair $(A',B')$ is $(\varepsilon/\alpha,p)$-regular with $p$-density at least $d-\varepsilon$.

      \item There are at most $\varepsilon |A|$ vertices in $A$ with less then $(d-\varepsilon)p|B|$ neighbours in $B$.
      \end{enumerate}
      
\end{lemma}
\noindent A partition $V(G)=V_0\cup V_1\cup\dots \cup V_k$ is said to be $(\eps,p)$-regular if
 \begin{enumerate}
	\item $|V_0|\leq \varepsilon|V(G)|$,
	\item $|V_i|=|V_j|$ for all $i,j\in[k]$, and
	\item all but at most $\varepsilon k^2$ pairs $(V_i,V_j)$ are $(\varepsilon,p)$-regular.
\end{enumerate}
We may now state a sparse version of Szemerédi's regularity lemma, due to Kohayakawa and Rödl~\cite{reglemma, rodlregularity}  .

\begin{theorem}\label{reg}Given $\varepsilon>0$ and $k_0\in \mathbb{N}$, there are $\eta>0$ and $K_0\ge k_0$ such that the following holds. Let $G$ be an $\eta$-upper-uniform graph on $n\ge k_0$ vertices and let $p\in (0,1)$. Then $G$ admits an $(\eps,p)$-regular partition $V(G)=V_0\cup V_1\cup \dots \cup V_k$ with $k_0\le k\le K_0$.\end{theorem}
Let $G$ be a graph that admits an $(\eps,p)$-regular partition $V(G)=V_0\cup V_1\cup\dots \cup V_k$, and let $d\in(0,1)$. The $(\eps,p,d)$-reduced graph $R$, with respect to this $(\eps,p)$-regular partition of $G$, is the graph with vertex set $V(R)=\{V_i:i\in[k]\}$, called clusters, such that $V_iV_j$ is an edge if and only if $(V_i,V_j)$ is an $(\eps,p)$-regular pair with $d_p(V_i,V_j)\ge d$. Next proposition establishes that the edge density of $R$ is roughly the same as in $G$. Since its proof is fairly standard in the applications of the Regularity Lemma, we omit it.

\begin{prop}\label{reg:edges}Let $\eps,\eta,p,d\in(0,1)$ and let $k\in\mathbb N$ such that $k\ge 1/\eps$. Let $G$ be an $(\eta,p)$-upper uniform graph on $n$ vertices that admits an $(\eps,p)$-regular partition $V(G)=V_0\cup V_1\cup\dots\cup V_k$, and let $R$ be the $(\eps,p,d)$-reduced graph of $G$ with respect to this partition. Then
	\[e(R)\ge \frac{e(G)}{(1+\eta)p}\left(\frac{k}{n}\right)^2-(6\eps+d)k^2.\]
	\end{prop}

\subsection{Facts about the random graph}

We state three lemmas concerning properties of $G(N,p)$ and we omit their proofs. The first two follow by a simple application of Chernoff's bound and the third by Janson's inequality.

\begin{lemma}\label{etap}
For every $\eta>0$ there exists $C>0$ such that if $p\geq C/N$ then, with high probability, $G(N,p)$ is $(\eta,p)$-uniform.
\end{lemma}
In particular, since any spanning subgraph of an $(\eta,p)$-uniform graph is $(\eta,p)$-upper-uniform, then, with high probability, every spanning subgraph of $G(N,p)$ is $(\eta,p)$-upper-uniform, as long as $p\ge C/N$.

\begin{lemma}  For every $\gamma>0$,
   $G=G(N,p)$ a.a.s satisfies the following properties.
\label{p2}
\begin{enumerate}[$(i)$]
 \item\label{lem:Gnp1} For every set $U\subseteq V$ with $|U|\geq \gamma N$,
    there are at most $64/\gamma p$ vertices in $V$ with less
    than $\gamma pN/8$ neighbours in $U$.
    \item\label{lem:Gnp2} For every $c>0$, there exists $0<c'<1$ such that $G$ is a weak $(c/p,c'N)$-expander. Moreover, $c'\rightarrow 0$
    as $c\rightarrow \infty$.
\end{enumerate}
\end{lemma}

\begin{lemma}\label{p3} For every $\gamma>0$ there exists $C'>0$ such that if $p\geq C'N^{-2/(r+2)}$, then $G=G(N,p)$ with high probability has the following property. For every $v\in V(G)$ and any $r$ disjoint sets $W_1,\dots,W_r\subseteq N(v)$, with $|W_i|\geq \gamma pN$ for each $i\in[r]$, there exists a copy of $K_{r+1}$ containing $v$ and one vertex in each $W_i$, for $i\in[r]$.
\end{lemma}

\section{Global Resilience of Large Trees}\label{sec:resilience}
This section is devoted to the global resilience of trees of linear size and bounded maximum degree in $G(N,p)$. We will prove the following result, which is a strengthening of Theorem 1.3.

\begin{theorem}\label{Resilience}Let $\delta,\varrho\in(0,1)$ and $D\ge 2$. There are positive constants $n_0,\eta_0$ and $C $ such that for all $0<\eta\le\eta_0$ and $n\ge n_0$ the following holds. Let $G$ be a $(\eta,p)$-uniform graph on $n$ vertices and let $p\in[0,1]$ with $pn\geq C$. Then every subgraph $G'\subseteq G$
with $e(G')\ge \left(\varrho+\delta \right)e(G)$ is $\mathcal{T}(\varrho n,D)$-universal.
\end{theorem}
It turns out that Theorem~\ref{resilience} easily follows from Theorem~\ref{Resilience}. Indeed, recall from Lemma~\ref{etap} that, with high probability, $G=G(N,p)$ is $(\eta_0,p)$-uniform for $p\ge C/N$ and therefore, by Theorem~\ref{Resilience}, any subgraph $G'\subseteq G$ with $e(G')\ge (\varrho +\delta)e(G(N,p))$ is $\mathcal T(\varrho N,D)$-universal. 

\subsection{Cutting up a tree} Now we show how to cut a given tree $T$ into a constant number of small rooted subtrees, such that the root of each of these subtrees is at even distance from the root of $T$. This partition will be a straightforward modification of the following result proved by Balogh, Csaba and Samotij~\cite[Lemma 15]{Samotij}.
\begin{lemma}
	\label{cutS}
	Let $D\ge 2$ and let $(T,r)$ be a rooted tree with maximum degree at most $D$. If $\beta\ge 1/|V(T)|$, then there exists a family of $t\le 4/\beta$ disjoint rooted subtrees $(T_i,r_i)_{i\in [t]}$ such that $V(T)=V(T_1)\cup \dots\cup V(T_t)$ and for each $i\in [t]$ we have 
	\begin{enumerate}
	    \item\label{lem:cutS1}$|V(T_i)|\le D^2\beta |V(T)|$,
	    \item\label{lem:cutS2} $T_i$ is connected (by an edge) to at most $D^3$ other subtrees, and 
	    \item \label{lem:cutS3}$T_i$ is rooted at $r_i$ and all the children of $r_i$ belong to $T_i$.
	    \end{enumerate}
\end{lemma}
Given a tree $T$, let $(T_i,r_i)_{i\in[t]}$ be the family given by Lemma~\ref{cutS}. We may define an auxiliary graph $T_\Pi$, called \textit{cluster tree}, with vertex set $V(T_\Pi)=[t]$ and edge set $E(T_\Pi)=\{ij\mid \text{ $T_i$ and $T_j$ are adjacent in $T$}\}$.
\begin{figure}[h!]
    \centering
    \includegraphics[scale=.6]{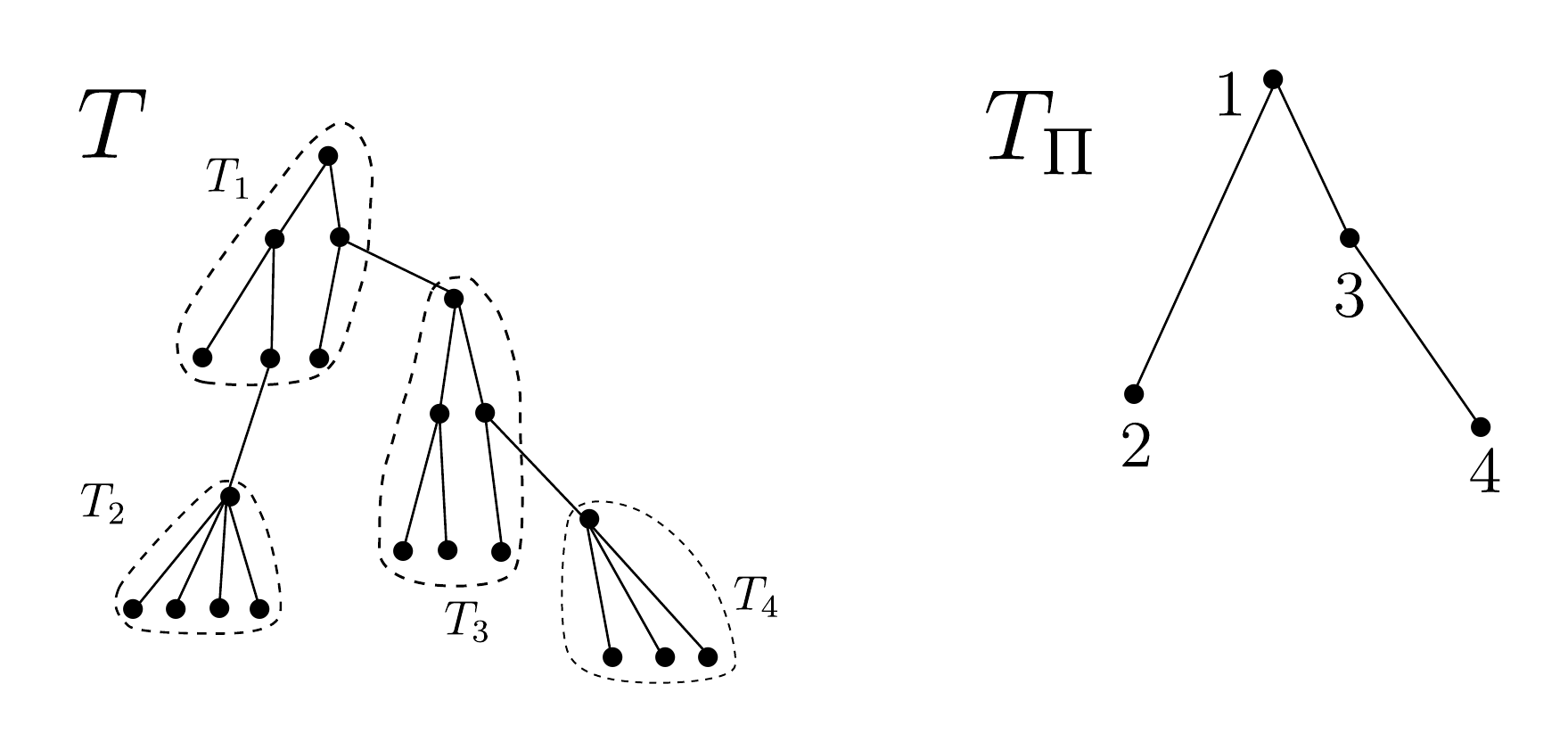}
    \caption{Cluster tree}
\end{figure}

Note that property~\textit{(\ref{lem:cutS1})} of Lemma~\ref{cutS} implies that $|V(T_\Pi)|\leq 4/\beta$, and property~\textit{(\ref{lem:cutS2})} implies that $\Delta(T_\Pi)\leq D^3$, which plays a crucial role in the embedding strategy. We only need to refine the partition given by Lemma~\ref{cutS} in order to impose that the root of each subtree is at even distance from the root of $T$, which is a stronger property than \textit{(\ref{lem:cutS3})}.
\begin{prop}
	\label{cutP}
	Let $D\ge 2$ and let $(T,r)$ be a rooted tree with maximum degree at most $D$. If $\beta\ge 1/|V(T)|$, then there exists a family of $t\le 4D/\beta$ disjoint rooted subtrees $(T_i,r_i)_{i\in [t]}$ such that $V(T)=V(T_1)\cup \dots\cup V(T_t)$ and for each $i\in[t]$ we have
	\begin{enumerate}
	    \item $|V(T_i)|\le D^4\beta |V(T)|$,
	    \item $T_i$ is rooted at $r_i$ and the distance from $r_i$ to $r$ is even,
	    \item all the children of $r_i$ belong to $T_i$, and
	    \item the corresponding cluster tree has maximum degree at most $D^4$.
	\end{enumerate}
\end{prop}

\begin{proof}
	Starting with the partition given by Lemma~\ref{cutS}, we will refine this partition as we run a breadth first search (BFS) on $(T,r)$. Suppose that in this search we have reached a vertex $v$, which is the root of a subtree in the current partition, such that $v$ and every root before $v$ in the search are at even distance from each other in the current partition.
	
	If there is a root $u$ of some subtree in the current partition, which is at odd distance from $v$ and such that the subtree rooted at $v$ is adjacent to $u$, then we may update the partition by splitting the tree rooted at $u$ (each neighbour of $u$ is now the root of a subtree) and adding $u$ to the subtree rooted at $v$. We repeat this process for every such $u$.
	Note that after these splittings, the root of each tree that is adjacent to the tree rooted at $v$ is at even distance from all the previous roots. Moreover, a subtree of the original partition can only be split by this process when the BFS reaches its parent. Since each subtree has only one parent, they are split at most once into $D$ new subtrees and therefore the final partition has at most $4D/\beta$ new subtrees. For the same reason, the maximum degree of the cluster tree cannot go higher than $D^4$, since the original $T_\Pi$ had maximum degree at most $D^3$.

Finally, the size of each subtree grows by at most $D^3$ if the roots of its children are added. Since the update only moves forward in the BFS order, at the end of the process each subtree has size at most $D^2\beta|V(T)| + D^3 \leq D^4\beta |V(T)|$.
\end{proof}
\subsection{Structure in the reduced graph} In this subsection, we will follow a strategy inspired in the approach of Besomi, Stein and the third author~\cite{BPS3} to the Erd\H os--S\'os Conjecture for bounded degree trees and dense host graphs. We will prove that if $H$ is an $(\eta,p)$-upper-uniform graph with  $e(H)\ge(\varrho+\delta/2)pn^2/2$, then $H$ has an $(\eps,p,d)$-reduced graph $R$ with a useful substructure. That is, $R$ contains a cluster $X$ of large degree such that the neighbourhood of $X$ can be partitioned as $N(X)=V(\mathcal M)\cup \mathcal Y$, where $\mathcal M$ is a matching and $\mathcal Y$ is an independent set. Moreover, denoting by $\mathcal H$ the bipartite graph induced by $\mathcal Y$ and $\mathcal Z=N(\mathcal Y)\setminus (X\cup N(X))$, either $\mathcal M$ is large enough or every cluster in $\mathcal Y$ has large degree in $\mathcal H$. In order to find such a structure, we need the following lemma (see~\cite[Lemma~3.5]{BPS1} for a proof).
\begin{lemma}\label{lem:mat_trian}Given a graph $F$, there exists an independent set $I$, a matching $M$ and a family of triangles $\Gamma$, such that $V(F)=I\cup V(M)\cup V(\Gamma)$. Moreover, we may write $V(M)=M_1\cup M_2$, where each edge $e\in M$ is of the form $e=v_1v_2$ with $v_i\in M_i$ for $i\in\{1,2\}$, so that $N(I)\subseteq M_1$. 
	\end{lemma}

\begin{prop}
  \label{matching} Let $\eps,\delta,\varrho\in(0,1)$ and let $d=\delta/100$. There exist $n_0,k, K_0\in\mathbb N$ and $n_0>0$ such that for all $0<\eta\le\eta_0$, $p\in (0,1)$ and $n\ge n_0$ the following holds. Every $(\eta,p)$-upper uniform $n$-vertex graph $H$, with $2e(H)\ge(\varrho+\delta/2)pn^2$, admits an $(\eps,p)$-regular partition with $1/\eps\le k\le K_0$ parts such that its $(\eps,p,d)$-reduced graph $R$ satisfies the following. There exist $X\in V(R)$, a matching $\mathcal{M}$, and a bipartite induced subgraph $\mathcal H=R[\mathcal Y, \mathcal Z]$ such that
  \begin{enumerate}[$(a)$]
    \item\label{R:i} $N(X)=V(\mathcal{M})\cup \mathcal {Y}$ and $V(\mathcal M)\cap\mathcal Y=\emptyset$;
    \item\label{R:ii} $|V(\mathcal{M})|+ |\mathcal{Y}|\geq \left( \varrho + {\delta}/{3}\right)k$; and
    \item\label{R:iii} for all $Y\in \mathcal{Y}$ we have
\[	|N_\mathcal{H}(Y)|\geq \left( \varrho + \frac{\delta}{4}\right)\dfrac{k}{2} - \dfrac{|V(\mathcal{M})|}{2}.\]\end{enumerate}
\end{prop}

\begin{proof} Given $\eps'=\min\{\eps/5,\delta/1000\}$ and $k_0=1/\eps'$, let $\eta_0,n'_0$ and $K'_0$ be the outputs of the regularity lemma (Theorem~\ref{reg}) with parameters $\eps'$ and $k_0$. Setting $n_0=n'_0$ and $\eta_0=\min\{\eta'_0,\delta/1000\}$, let $H$ be an $(\eta,p)$-upper uniform graph on $n\ge n_0$ vertices and $0<\eta\le \eta_0$. Then $H$ admits an $(\eps',p)$-regular partition $V(H)=V_0'\cup V_1'\cup\dots\cup V'_\ell$, with $1/\eps'\le \ell\le K_0$, and let us denote by $R'$ the $(\eps',p,2d)$-reduced graph of $H$ with respect to this regular partition. By Proposition~\ref{reg:edges} and the bound on $e(H)$ we have
	\begin{equation}\label{R:avg}e(R')\ge (1+\eta)^{-1}\Big(\varrho+\frac\delta 2\Big)\frac{\ell^2}2-(6\eps'+2d)\ell^2\ge\Big(\varrho+\frac{\delta}{3}\Big)\frac{\ell^2}2.
	\end{equation}
Note that \eqref{R:avg} implies that the average degree of $R'$ is at least $(\varrho+\delta/3)\ell$. Thus, by successively removing vertices of low degree, we may find a subgraph $R_0\subseteq R'$ such that
\[d(R_0)\ge \Big(\varrho+\frac{\delta}{3}\Big){\ell}\hspace{.5cm}\text{and}\hspace{.5cm}\delta(R_0)\ge \Big(\varrho+\frac{\delta}{3}\Big)\frac{\ell}2.\]
In particular, this implies that there exists a cluster $X'\in V(R_0)$ with degree at least $(\varrho + \delta/3)\ell$ in $R_0$. Applying Lemma~\ref{lem:mat_trian} to $N_{R_0}(X')$, we find an independent set $I$, a matching $\mathcal{M'}$ and a collection of triangles $\Gamma$ that partition $N_{R_0}(X')=I\cup V(\mathcal M')\cup V(\Gamma)$, and moreover, by writing $V(\mathcal M')=M_1\cup M_2$ we have that $N_{R_0}(I)\subseteq M_1$. Note that the minimum degree on $R_0$ implies that for all $Y\in I$ we have
  \begin{equation}\label{Y:deg}|N_{R_0}(Y)\setminus (X'\cup N_{R_0}(X))|\ge \left(\varrho + \dfrac{\delta}{3}\right)\dfrac{\ell}{2} - 1 - \dfrac{|V(\mathcal{M})|}{2} \geq \left(\varrho + \dfrac{\delta}{4}\right)\dfrac{\ell}{2} - \dfrac{|V(\mathcal{M})|}{2}.\end{equation}
We aim to simplify this structure by considering a blow-up of $R$. We then consider, for each $i\in [\ell]$, an arbitrary partition $V_i=V_{i,0}\cup V_{i,1}\cup V_{i,2}$ so that $|V_{0,i}|\le 1$ and $|V_{i,1}|=|V_{i,2}|$. 
Note that for every $i\in[\ell]$ we have that $|V_{i,1}|=|V_{i,2}|\geq |V_i|/3$. Therefore, Lemma~\ref{lem:regularpairs} implies that, for every $i\in [\ell]$ with $V_iV_j\in E(R')$ and $a,b\in\{1,2\}$, the pair $(V_{i,a},V_{j,b})$ is $(\eps,p)$-regular with density at least $d$. Moreover, by setting $V_0=V_0'\cup V_{1,0}\cup\dots\cup V_{\ell,0}$ we conclude that $V(H)=V_0\cup V_{1,2}\cup V_{2,2}\cup\dots \cup V_{\ell,1}\cup V_{\ell,2}$ is an $(\eps,p)$-regular partition with $2\ell+1$ parts. Let $R$ be the $(\eps,p,d)$-reduced graph of $H$ with respect to this partition, and let $k=2\ell$ be the number of vertices of $R$ (note that $R$ contains a $2$-blowup of $R'$). 

Let $X$ be one of the clusters coming from $X'$, and $\mathcal Y$ be the set of all  the $V_{i,a}$ such that $V'_i\in I$ and $a\in\{1,2\}$. Note that $K_{2,2,2}$ and $K_{2,2}$, the $2$-blowups of a triangle and of an edge respectively, each contains a perfect matching. Therefore the set $\{V_{i,a}: V_i\in \mathcal{M}'\cup \Gamma\text{ , } a\in\{1,2\}\}$ contains a perfect matching in $R$, which we denote by $\mathcal{M}$. Let $\mathcal Z=N_R(\mathcal Y)\setminus (X\cup N_R(X))$ and let $\mathcal H$ as the bipartite graph induced by $\mathcal Y$ and $\mathcal Z$. It is straightforward to check that $X$, $\mathcal M$ and $\mathcal H$ satisfy \eqref{R:i} and \eqref{R:ii} and that \eqref{R:iii} follows from~\eqref{Y:deg}.

\end{proof}

  \subsection{Proof of Theorem \ref{Resilience}}
  
   As we mentioned in the sketch of the proof, the idea is to use the structure given by Proposition~\ref{matching}, that is, the cluster $X$, the matching $\mathcal M$ and the bipartite graph $\mathcal H$. To do so, we first need to cut the tree into a family $(T_i,r_i)_{i\in [t]}$ of tiny subtrees such that the root of all the subtrees are in the same colour class (see Proposition~\ref{cutP}). The main challenge in the proof is the assignment of each $T_i$ to some edge of $\mathcal M\cup \mathcal H$ into which it will be embedded. After this, we remove some bad vertices from each cluster so that each subtree $T_i$ is assigned to a pair $(Y_{i,1},Y_{i,2})$ which induces a bipartite expander graph and that connects well with a large subset of $X$ (see Claim~\ref{res:claim}). Finally, by using an embedding tool due to Balogh, Csaba and Samotij~\cite[Corollary~12]{Samotij}, we embed each subtree into the pair that was assigned to that tree. We state this result below.
  \begin{lemma}
  	\label{magmar}
  	Let $D\geq 2$ and let $H$ be a bipartite graph with bipartition classes $V_1$ and $V_2$, where $|V_1|\leq |V_2|$. Suppose that $H$ is a bipartite $(m,D+1)$-expander with $0<m<|V_1|/(2D+1)$. Then $H$ contains all trees $T$ with maximum degree at most $D$ and bipartition classes $A_1$ and $A_2$ such that $|A_1|\leq|V_1|-(2D+1)m$ and $|A_2|\leq |V_2|-(2D+1)m$. Furthermore, for every $i\in\{1,2\}$, $u\in A_i$ and $v\in V_i$ there exists an embedding $\varphi:V(T)\rightarrow H$ such that $\varphi(u)=v$.
  \end{lemma}
  Although it is not true that $(\eps,p)$-regular pairs are bipartite expanders (for example they can have isolated vertices), any large subgraph of an $(\eps,p)$-regular pairs contains an almost spanning subgraph which is a bipartite expander. The following lemma was proved in \cite[Lemma~19]{Samotij}, and its proof is similar to that of Proposition~\ref{indhaxxel}.
  
  \begin{lemma}
	\label{pikachu}
	Let $(A,B)$ be an $(\varepsilon,p)$-regular pair such that $d_p(A,B)>\varepsilon$. Suppose that $|A|=|B|=m$ and let $A'\subseteq A$ and $B'\subseteq B$ be sets of size at least $(4D+6)\varepsilon m$. Then there are subsets $A''\subseteq A'$ and $B''\subseteq B'$ such that
	\begin{enumerate}[(a)]
		\item $|A'\setminus A''|\leq \varepsilon m$ and $|B'\setminus B''|\leq \varepsilon m$, and
		\item the subgraph induced by $(A'',B'')$ is a bipartite $(\varepsilon m,2D+2)$-expander.
	\end{enumerate}
\end{lemma}
  
 Now we are ready to prove Theorem~\ref{Resilience}.
  
\begin{proof}[Proof of Theorem \ref{Resilience}]
Let $n'_0, k, K_0$ and $\eta_0$ be the outputs of Proposition~\ref{matching} with inputs $\delta,\varrho$ and $\varepsilon={\delta^4}/{(2^{28}D^6)}$. We set
\begin{equation}\label{chikorita}\beta=\dfrac{\delta^2}{2^{12}kD^4}\hspace{.5cm}\text{and}\hspace{.5cm}C_0=\frac{2^{17}10^2D^5K_0^2}{\delta^3},  
\end{equation}
and let $n_0=\max\{n_0',\beta^{-1}\}$ and $n\ge n_0$. Given $p\ge C_0/n$ and $0<\eta\le\eta_0$, let $G$ be an $(\eta,p)$-uniform graph on $n$ vertices and let $G'\subseteq G$ be a subgraph with 
\begin{equation*}2e(G')\ge (\varrho+\delta)2e(G)\ge (1-\eta)(\varrho+\delta)pn^2\ge \left(\varrho+\frac\delta 2\right)pn^2.
\end{equation*}
Since $G'$ is $(\eta,p)$-upper uniform, by  Proposition~\ref{matching} we may find an $(\eps,p)$-regular partition $V(G')=V_0\cup V_1\cup\dots\cup V_k$, with $1/\eps\le k\le K_0$, such that the $(\eps,p,\delta/100)$-reduced graph $R$, with respect to this partition, contains a cluster $X$, a matching $\mathcal{M}$ and a bipartite subgraph $\mathcal H$, with vertex set $V(\mathcal H)=\mathcal Y\cup \mathcal Z$, satisfying the conclusions of Proposition~\ref{matching}.

Let $T\in \mathcal{T}(\varrho n,D)$ be given. We consider the bipartition of $T$ that assigns colour $1$ to the smaller partition class of $T$ and colour $2$ to the larger one, and then we  choose an arbitrary vertex $r$ in colour $1$ as the root of $T$. We apply Proposition~\ref{cutP} to $(T,r)$, with parameter $\beta$,  obtaining a family $(T_i,r_i)_{i\in [t]}$ of $t\leq 4D/\beta$ rooted trees, each of size at most $D^4\beta\varrho n$. Furthermore, each root $r_i$ is at even distance from $r$ and therefore every root has colour $1$. For $i\in[t]$, let us write $T_{i,j}$ for the set of vertices of $T_i$ having colour $j\in\{1,2\}$. 

Let $m$ denote the size of the clusters and observe that $m\ge (1-\eps)n/k$. The heart of the proof is embodied by the following claim.  
 \begin{claim}\label{res:claim}
  For each $i\in [t]$, there are sets $(Y_{i,1},Y_{i,2})$ and $W_i\subseteq X$ such that the following holds.
 \begin{enumerate}[$(1)$]
 	\item\label{claim:res 2} The sets $\{Y_{i,j}: (i,j)\in [\ell]\times \{1,2\} \}$ are pairwise disjoint and disjoint from $X$
 	\item\label{claim:res 3}  $|Y_{i,j}| \ge  |T_{i,j}| + 13D\varepsilon m$, for each $j\in\{1,2\}$.
  \item\label{claim:res 4} $G'[Y_{i,1},Y_{i,2}]$ is a bipartite $(\varepsilon m, 2D+2)$-expander.
 \item\label{claim:res 5} Every vertex of $Y_{i,2}$ has at least ${\delta}pm/(200)$ neighbours in $W_i$.
 \item\label{claim:res 6} If $T_\ell$ is a child of $T_i$ in the cluster tree, then every vertex of $W_i$ has at least $D+1$ neighbours in $Y_{\ell,2}$.
 \end{enumerate}
 \end{claim}	
Before proving Claim~\ref{res:claim}, let us show how to derive Theorem~\ref{Resilience} from it. Assume that we have ordered $[t]$ so that if $T_\ell$ is below $T_i$, with respect to the root of $T$, then $i\le \ell$. Starting with the subtree containing $r$, we will embed $(T_i)_{i\in[t]}$ following this ordering. Let us denote by $\varphi$ the partial embedding of $T$. For every embedded subtree $(T_i,r_i)$ we will ensure that
\begin{enumerate}[$(a)$]
	\item\label{res:emb1} $\varphi(r_i)\in W_s$ for some $s\le i$, and
	\item\label{res:emb2} $\varphi(T_{i,j}\setminus \{r_i\})\subseteq Y_{i,j}$ for $j\in \{1,2\}$. 
\end{enumerate}
Suppose we are about to embed a subtree $T_\ell$ which is a child of some subtree $T_i$ that was already embedded satisfying~\eqref{res:emb1} and ~\eqref{res:emb2}. Let $v_i\in V(T_i)$ be the parent of $r_\ell$ and note that $v_i$ is embedded into some vertex $\varphi(v_i)\in Y_{i,2}$ (since $v_i$ is adjacent to $r_\ell$ and every root has colour $1$).
\begin{figure}[ht]
	\centering
	\includegraphics[scale=.7]{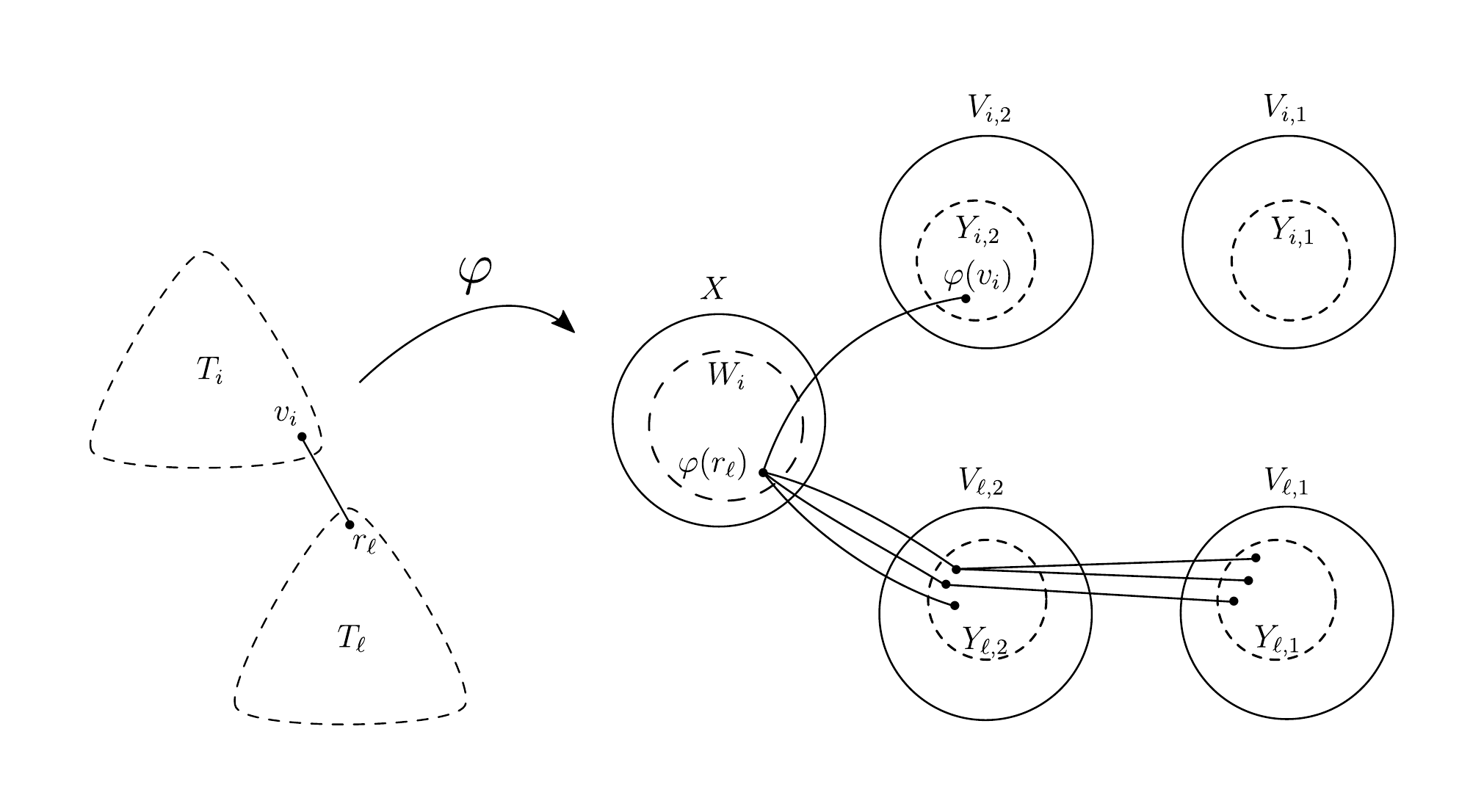}
	\caption{Embedding of $T_\ell$}
\end{figure}
Then, because of Claim~\ref{res:claim}~(\ref{claim:res 5}) 
\[| W_i\cap N_{G'}(\varphi(v_i))|\ge \dfrac{\delta}{200}pm\ge(1-\eps)\frac{\delta C_0}{200k}\ge \dfrac{8D}{\beta}\ge 2t\]
and therefore  at least one neighbour of $\varphi(v_i)$ has not been used during the embedding. We choose any unused vertex $w_\ell\in W_i\cap N_{G'}(\varphi(v_i))$ and set $\varphi(r_\ell)=w_\ell$ (when we embed $T_1$, we choose any vertex vetex $w_1\in W_1$ as the image of $r_1=r$). By Claim~\ref{res:claim}~(\ref{claim:res 4}) we know that $G'[Y_{i,1},Y_{i,2}]$ is a bipartite $(\eps m, 2D+2)$-expander, we will prove now that $G'[Y_{\ell,1}\cup \{w_\ell\}, Y_{\ell,2}]$ is a bipartite $(\varepsilon m +1,D+1)$-expander.

Indeed, since $G'[Y_{i,1},Y_{i,2}]$ is a bipartite $(\eps m, 2D+2)$-expander is easy to see that the expansion conditions hold for every subset $X$ of $Y_{\ell,1}$ or of $ Y_{\ell,2}$. Let $X'\subseteq Y_{\ell,1}$ be non-empty and let us consider $X=X'\cup\{w_\ell\}$. If $|X'|\le \eps m$, then we have 
\[|N_{G'}(X)\cap Y_{\ell,2}|\ge (2D+2)|X'|\ge (D+1)|X|,\]
where the first inequality follows because $G'[Y_{\ell,1},Y_{\ell,2}]$ is bipartite $(\eps m,2D+2)$-expander. Similarly, if $|X'|\ge \eps m$ then we have
\[|N_{G'}(X)\cap Y_{\ell,2}|\ge |N_{G'}(X')\cap Y_{\ell,2}|\ge |Y_{\ell,2}|-(\eps m+1).\]
Finally, if $X=\{w_\ell\}$ then by Claim~\ref{res:claim}~(\ref{claim:res 6}) we know that $|N_{G'}(w_\ell)\cap Y_{\ell,2}|\ge D+1$, and therefore $G'[Y_{\ell,1}\cup \{w_\ell\}, Y_{\ell,2}]$ is a bipartite $(\varepsilon m +1,D+1)$-expander.

Lemma~\ref{magmar} allows us to embed trees with bipartition classes of size $|Y_{\ell,j}|- (2D+1)(\varepsilon +1)$. We combine this information with property~(\ref{claim:res 3}) of Claim~\ref{res:claim} to get that
\begin{equation*}
	|Y_{\ell,j}| - (2D+1)(\varepsilon m +1)\geq |T_{\ell,j}| + 13D\varepsilon m - 6D\varepsilon m \geq|T_{\ell,j}|
\end{equation*}
for each $j\in\{1,2\}$. Therefore, Lemma~\ref{magmar} yields an embedding of $T_\ell$ into $(Y_{\ell,1}\cup\{w_\ell\}, Y_{\ell,2})$ so that $\varphi(T_{\ell,j}\setminus \{r_\ell\})\subseteq Y_{\ell,j}$ for $j\in\{1,2\}$ and $r_\ell$ is mapped to $w_\ell$. We finish by remarking that Claim~\ref{res:claim}~(\ref{claim:res 2}) ensures that this embedding $T_\ell$ does not intersect the previously embedded subtrees.\end{proof}
\begin{proof}[Proof of Claim~\ref{res:claim}]
 Let $\sigma$ be a permutation on $[t]$ such that for all $1\le i<j\le t$ we have
\begin{equation*}|T_{\sigma(i),2}|-|T_{\sigma(i),1}|\ge |T_{\sigma(j),2}|-|T_{\sigma(j),1}|.
\end{equation*}
We chose colour $2$ to be the larger class of $V(T)$ so that for every $\ell\in [t]$ we have
\begin{equation}
\label{butterfree}
	\sum_{i=1}^{\ell} (|T_{\sigma(i),2}|-|T_{\sigma(i),1}|) \geq 0.
\end{equation}
The proof of Claim~\ref{res:claim} will be done in two stages. In the first stage, for each $i\in[t]$ the subtree $T_i$ will be assigned to a pair of sets $(X_{i,1},X_{i,2})$, contained in some edge from $\mathcal M\cup E(\mathcal H)$, such that  $|X_{i,j}|=|T_{i,j}|+16D\eps m$ for $j\in\{1,2\}$. In the second stage, we will remove some vertices from each set in order to find the sets $W_i\subseteq X$ and $Y_{i,j}\subseteq X_{i,j}$ satisfying the properties $(1)-(5)$ from Claim~\ref{res:claim}.\\

\noindent\textbf{Stage 1 (Assignation):} In this stage we will prove that for each $i\in [t]$, there exists an edge $V_{i,1}V_{i,2}\in \mathcal{M} \cup E(\mathcal H)$ and sets $X_{i,j} \subseteq V_{i,j}$, for $j\in\{1,2\}$,  such that 
	\begin{enumerate}[$(A)$]
		\item\label{assig:1} $X_{i,j}\cap X_{\ell,j'}=\emptyset$ if $\{i,j\}\not=\{\ell,j'\}$;
		\item\label{assig:2} $|X_{i,j}| = |T_{i,j}| + 16D\varepsilon m$; and
		\item\label{assig:3} if $(V_{i,1},V_{i,2})\in E(\mathcal H)$ then $V_{i,2} \in \mathcal{Y}$.
	\end{enumerate}
The assignment will be done in two steps following the order given by $\sigma$. At step 1 we assign trees to edges from $\mathcal H$ until we use a large proportion of $\mathcal Y\cup\mathcal Z$, and at step 2 we will use  edges from $\mathcal M$ ensuring that the clusters from each edge of $\mathcal M$ are used in a balanced way. \\

 \noindent\textbf{Step 1:} We will assume that $|\mathcal M|\le (\varrho+\delta/16)k$, as otherwise we just skip this step. Let us set $Q=(\varrho+\delta/4)k-|V(\mathcal M)|$ and note that we have 
 \begin{equation*}|\mathcal Y|\ge Q\ge \frac \delta{16}k\hspace{.5cm}\text{and}\hspace{.5cm}d_\mathcal H(Y)\ge Q/2\text{ for all $Y\in\mathcal Y$.}\end{equation*}
We will choose sets in $\mathcal{Y}\cup \mathcal{Z}$ until we have assigned at least $(1-\delta/16)Qm$ vertices to $\mathcal Y\cup\mathcal Z$. Following the order of $\sigma$, assume that we have made the assignation up to some $0\le\ell\le t-1$ and we are about to assign the tree $T_{\sigma(\ell+1)}$. Suppose that there are $Y\in \mathcal{Y}$ such that
\begin{equation}\label{squirtle}\sum_{X_{\sigma(i),2}\subseteq Y} |X_{\sigma(i),2}| \leq m - (D^4\beta n + 16D\varepsilon m), \end{equation}

\noindent and $Z\in N_\mathcal H(Y)$ with
\begin{equation}\label{charmander}\sum_{X_{\sigma(i),1}\subseteq Z} |X_{\sigma(i),1}| \leq m - (D^4\beta n + 16D\varepsilon m).\end{equation}

\noindent Since $|T_{\sigma(\ell+1)}|\le D^4\beta\varrho n$, we can select sets  $X_{\sigma(\ell+1),1}\subseteq Z$ and $X_{\sigma(\ell+1),2}\subseteq Y$, disjoints from the previously chosen sets, such that $|X_{\sigma(\ell+1),j}| = |T_{\sigma(\ell+1),j}| + 16D\varepsilon m$ for $j\in\{1,2\}$. So, if there is no $Y\in\mathcal Y$ satisfying~\eqref{squirtle}, then we have
\begin{equation*}
	\begin{split}
		\sum_{i=1}^\ell |T_{\sigma(i)}| \geq \sum_{i=1}^\ell |T_{\sigma(i),2}| &= \sum_{i=1}^\ell\big(|X_{\sigma(i),2}| - 16D\varepsilon m\big)\\
					       & \geq |\mathcal{Y}|m  - t\cdot 16D\varepsilon m - k\cdot (D^4\beta n + 16D\varepsilon m)\\
					       & \ge|\mathcal Y|m-\frac{\delta^2}{16^2}km\\
					    &\ge \Big(1-\frac\delta{16}\Big)Qm.
\end{split}
\end{equation*}
\noindent This means that we have already used enough vertices from $\mathcal Y\cup\mathcal Z$. On the other hand, if every $Y$ satisfying~\eqref{squirtle} has no neighbours satisfying~\eqref{charmander}, we may use \eqref{butterfree} to deduce 
\begin{equation*}
\label{beedrill}
	\begin{split}
		\sum_{i=1}^\ell |T_{\sigma(i)}| \geq 2\sum_{i=1}^\ell |T_{\sigma(i),1}| &= 2\sum_{i=1}^\ell\big(|X_{\sigma(i),1}| - 16D\varepsilon m\big)\\
		&\geq 2d_\mathcal H(Y)m - t\cdot 32D\varepsilon m - k\cdot 2(D^4\beta n + 16D\varepsilon m)\\
											    &\geq Qm  - \frac{\delta^2}{16^2}km\\
						&\ge \Big(1-\frac\delta{16}\Big) Qm.
	\end{split}
\end{equation*}
This means that if at step $\ell+1\in[t]$ we could not find a pair $(Y,Z)$ satisfying~\eqref{squirtle} and~\eqref{charmander}, then we have used vertices at least $(1-\delta/{16})Qm$ vertices from $\mathcal{Y}\cup\mathcal Z$ at step $\ell$. \\

\noindent\textbf{Step 2:} Let $0\le \ell_0\le t$ be such that $T_{\sigma(1)},\dots, T_{\sigma(\ell_0)}$ have been assigned to $\mathcal Y\cup\mathcal Z$, satisfying~\eqref{assig:1},\eqref{assig:2} and \eqref{assig:3}, and
\begin{equation}\label{charizard}\Big(1-\frac\delta{16}\Big)Qm\le\sum_{i=1}^{\ell_0} |T_{\sigma(i)}|\le \Big(1-\frac\delta{16}\Big)Qm+D^4\beta\varrho n.
\end{equation}
Assume that $\ell_0<t$, otherwise we are done. For $\ell_0+1\le i\le t$ we will assign each $T_{\sigma(i)}$ to some edge $AB\in \mathcal M$. At each step we will ensure that for every edge $AB\in\mathcal M$ we have
\begin{equation}\label{balance}
 \left|\sum_{X_{\sigma(i),j}\subseteq A} |X_{\sigma(i),j}| -\sum_{X_{\sigma(i),j}\subseteq B} |X_{\sigma(i),j}|\right| \leq D^4\beta\varrho n.   
\end{equation}Suppose we are about to assign a subtree $T_{\sigma(\ell)}$, for some $\ell\ge\ell_0+1$, and that~\eqref{balance} holds at step $i=\ell-1$ (note that~\eqref{balance} holds trivially at step $\ell_0$). Suppose that  there is an edge $AB\in\mathcal M$ such that
\begin{equation}\label{ratata}\max\Big\{\sum_{X_{\sigma(i),j}\subseteq A} |X_{\sigma(i),j}| , \sum_{X_{\sigma(i),j}\subseteq B} |X_{\sigma(i),j}|\Big\} \leq m - (D^4\beta\varrho n + 16D\varepsilon m). \end{equation}
We assume that the maximum is attained by the second term, that is to say that we have used more vertices in $B$ than in $A$. Let $j^\star=\argmax \limits_{j\in\{1,2\}}|T_{\sigma(\ell),j}|$ and then we may take sets
\begin{itemize}
 \item $X_{\sigma(\ell),j^\star}\subseteq A$ with  $|X_{\sigma(\ell),j^\star}| = |T_{\sigma(\ell),j^\star}| + 16D\varepsilon m$, and
 \item $X_{\sigma(\ell),3-j^\star}\subseteq B$ with  $|X_{\sigma(\ell),3-j^\star}|=|T_{\sigma(\ell),3-j^\star}| + 16D\varepsilon m.$
\end{itemize}
\noindent disjoints from the previously chosen sets. Note that we have assigned the larger colour class of $T_{\sigma(\ell)}$ to the less occupied cluster in $\{A,B\}$. Furthermore, since~\eqref{balance} holds at step $\ell-1$ and as $|T_{\sigma(\ell)}|\le D^4\beta \varrho n$, the assignment of $T_{\sigma(\ell)}$ implies that~\eqref{balance} holds at step $\ell$. So suppose that~\eqref{ratata} does not hold at step $\ell-1$ for any $AB\in\mathcal M$. Then we have 
\begin{equation*}\sum_{i=\ell_0+1}^{\ell-1}|T_{\sigma(i)}| \ge |V(\mathcal{M})|m - t\cdot 32D\varepsilon m - k\cdot (3D^4\beta \varrho n +32D\varepsilon m)\ge |V(\mathcal M)|m-\frac\delta{16}km\end{equation*}

\noindent that together with \eqref{charizard} yields
\begin{eqnarray*}
		\sum_{i=1} ^{\ell-1}|T_{\sigma(i)}|  \geq  \Big(1-\frac\delta{16}\Big)Qm +|V(\mathcal M)|m-\frac{\delta}{16}km
						& \geq & \Big(1-\frac\delta{16}\Big)\Big(\varrho+\frac{\delta}4\Big)km-\frac{\delta}{16}km\\	&\ge&\Big(\varrho+\frac\delta8\Big)km\\
						&\ge& \Big(\varrho+\frac\delta{16}\Big)n,
	\end{eqnarray*}

\noindent which is impossible since $|T|=\varrho n$. This implies that we can make the assignation for each $\ell\in[t]$.\\

\noindent\textbf{Stage 2 (Cleaning):} Assume that the cluster tree is ordered according to a BFS starting from the subtree which contains the root of $T$. Starting with a leaf of the cluster tree, suppose that we have found the sets $Y_{i,j}$ satisfying properties $(1)-(5)$ for all subtrees $T_i$ below $T_\ell$ in the order of the cluster tree. Let
\[W_\ell:=\{v\in X: d(v,Y_{i,2})\geq D+1\text{ for all } i \text{ such that $T_i$ is a child of $T_\ell$}\}.\]

\noindent We want to prove that $W_\ell$ has a reasonable size. Given a child $T_i$ of $T_\ell$ in the cluster tree, we have that
\[|Y_{i,2}|\geq |T_{i,j}|+13D\varepsilon m\geq (D+1)\varepsilon m\]

\noindent and therefore, since $(X,V_{i,2})$ is $(\varepsilon,p)$-regular, by Lemma~\ref{lem:regularpairs} there are at most $(D+1)\varepsilon m$ vertices in $X$ with less than $D+1$ neighbours in $Y_{i,2}$. Since the auxiliary tree has maximum degree $D^4$, then $W_\ell$ has at least
\[ |X| - (D+1)D^4\varepsilon |X| \geq \dfrac{m}{2}\]

\noindent vertices. Now, since $(X,V_{\ell,2})$ is $(\varepsilon,p)$-regular, then by Lemma~\ref{lem:regularpairs} the pair $(W_\ell,V_{\ell,2})$ is $(2\varepsilon,p)$-regular with $p$-density at least $\delta/(100)- \varepsilon$. By Lemma~\ref{lem:regularpairs} there are at most $2\varepsilon m$ vertices of $V_{ \ell,2}$ with less than
\[\left(\dfrac{\delta}{100} -3\varepsilon \right)p|W_\ell| \geq \dfrac{\delta}{200}pm\]

\noindent neighbours in $W_\ell$. We remove these vertices from $X_{\ell,2}$ to obtain a subset $X'_{\ell,2}\subset X_{\ell,2}$ such that every vertex in $X'_{\ell,2}$ has at least $\delta pm/200$ neighbours in $W_\ell$. Now, we need to find an expander subgraph of $(X_{\ell,1},X'_{\ell,2})$. %To do so we use the following lemma (see~\cite{Samotij} for a proof).
Since $(V_{\ell,1},V_{\ell,2})$ is $(\eps,p)$-regular with $d_p(V_{\ell,1},V_{\ell,2})\ge \delta/100$ and
\[|X_{\ell,1}|,|X'_{\ell,2}|\geq 16D\varepsilon m - 2\varepsilon m \geq (4D+6)\varepsilon m,\]
we use Lemma~\ref{pikachu} to  obtain a pair $(Y_{\ell,1},Y_{\ell,2})$, with $Y_{\ell,1}\subseteq X_{\ell,1}$ and $Y_{\ell,2}\subseteq X_{\ell,2}'$, such that $G'[Y_{\ell,1},Y_{\ell,2}]$ is bipartite $(\varepsilon m,2D+2)$-expander and  satisfies $|Y_{\ell,j}|\geq |X_{\ell,j}| - 3\varepsilon m\ge |T_{\ell,j}|+13D\eps m$ for $j\in\{1,2\}$.
\end{proof}
%%%%%%%%%%%%%%%%%%%%%%STABILITY%%%%%%%%%%%%%%%%%%%%%%%%%
\section{Proof of Theorem \ref{stab}}\label{sec:stab}

The proof of Theorem \ref{stab} follows from the following stability result.

\begin{theorem}
\label{finalstab}
For every $r,D\ge 2$ there exist $\delta,C,C'>0$ such that if $N\geq (1-\delta)rn$ and $p\geq C'N^{-2/(r+2)}$, then $G=G(N,p)$ with high probability has the following property. For every blue-red colouring of $E(G)$, at least one of the following holds:
\begin{enumerate}[a)]
    \item\label{stab:blue clique} $G$ contains a blue copy of $K_{r+1}$.
    \item\label{stab:red trees} $G$ contains a red copy of every $T\in \mathcal{T}(n,D)$.
    \item\label{stab:partite} There exists a partition $V(G)=V_0\cup V_1\cup\dots\cup V_r$, with $|V_0|\leq C/p$ and $|V_i|\leq n+C/p$ for each $i\in [r]$. Moreover,  all edges of $G[V_i,V_j]$ are coloured in blue for each $1\leq i<j\leq r$.
\end{enumerate}
\end{theorem}
Note that Theorem~\ref{finalstab} implies Theorem~\ref{stab}, as \eqref{stab:partite} cannot occur if $N> rn + (r+1)C/p$. Before proving Theorem~\ref{finalstab}, we will provide a rough structure of the colourings of a typical outcome of $G(n,p)$ by combining Theorems~\ref{resilience} and~\ref{erdsim}. 
\begin{prop}
\label{jolteon}
For every $\alpha,\eps>0$ and integers $r,D\ge 2$, there exist $C',\delta>0$ such that if $N\geq (1-\delta)rn$ and $p\geq C'N^{-2/(r+2)}$, then  $G=G(N,p)$ has, with high probability, the following property. For every blue-red colouring of $E(G)$, at least one of the following holds:
\begin{enumerate}[a)]
    \item $G$ contains a blue copy of $K_{r+1}$.
    \item $G$ contains a red copy of any $T\in \mathcal{T}(n,D)$.
    \item There exists a partition $V(G)=V_0\cup V_1\cup\dots\cup V_r$ such that $|V_0|\leq \alpha n$ and for each $i\in [r]$ we have $||V_i|- n|\leq \alpha n$ and $e_B(V_i)\leq \varepsilon pN^2$.
\end{enumerate}
\end{prop}

\begin{proof}
Without loss of generality, we may ask that $\varepsilon$ is small enough for calculations. Let $C'$ and $\delta'$ be the numerical outputs from Theorem~\ref{erdsim} with inputs $\eps$ and $r$. Let $\delta = \alpha/(2r^2)$, $\varrho = 1/r+2\delta$, $N\geq (1-\delta)rn$ and $p\geq C'N^{-2/(r+2)}$. Since $p\gg 1/N$, Theorem~\ref{resilience} implies that, with high probability, if $e(G_R)\geq (\varrho + \delta')e(G)$ then $G_R$ contains all trees with maximum degree $D$ and $\varrho N\geq n$ edges, and thus we may assume that
\[e(G_B)\geq\left(1-\dfrac{1}{r}-\delta'\right)e(G).\]
Theorem~\ref{erdsim} implies that, with high probability, all $K_{r+1}$-free subgraphs of $G$ with this many edges are $\varepsilon pN^2$-close to being $r$-partite. Therefore, we may assume that there exists a partition $V(G)=W_1\cup\dots\cup W_r$ such that $e_B(W_i)\leq \varepsilon pN^2$ for each $i\in [r]$. Since $p\gg 1/N$, we may also rule out the event in which $G$ is not $(\eta,p)$-uniform for some $0<\eta \ll \alpha$.
\begin{claim}\label{claim:stab:expander}
 In the events considered above, for each $i\in [r]$ the following holds. If $|W_i|\geq N/2r$, then there exists $V_i\subseteq W_i$, with $|W_i\setminus V_i|\leq \eta N$, such that $G_R[V_i]$ is a $(\eta N,\eta N,D)$-expander.
\end{claim}
\begin{proof}[Proof of Claim~\ref{claim:stab:expander}]
We prove first that $G_R[W_i]$ is a weak $(\eta N,\eta N)$-expander. Since $G$ is $(\eta ,p)$-uniform, then for every pair of disjoint sets $X,Y\subseteq V(G)$, with $|X|,|Y|\geq \eta N$, we have
\begin{equation*}
    e_R(X,Y)=e(X,Y)-e_B(X,Y) \geq \dfrac{p}{2}|X||Y| - \varepsilon pN^2 >0,
\end{equation*}
as long as $2\varepsilon < \eta^2 $. Since $|W_i|\geq (D+3)\eta N$, provided $\eta$ is small enough, we may apply Proposition~\ref{goodhaxxel} to find a set $V_i\subseteq W_i$, with $|W_i\setminus V_i|\leq \eta N$, such that $G_R[V_i]$ is an $(\eta N,\eta N,D)$-expander.
\end{proof}
For each $i\in[r]$ such that $|W_i|\geq N/2r$, by Claim~\ref{claim:stab:expander} we know that $G_R[V_i]$ is an $(\eta N,\eta N,D)$~-expander and then for all $X\subseteq V_i$, with $\eta N\leq |X|\leq 2\eta N$, we have
 \begin{equation*}
       |N_R(X)\cap V_i| \geq |V_i|-\eta N-|X|+1\geq (|V_i|-3D\eta N)+D|X|+1.
 \end{equation*}
Suppose that $V_1$ is the largest of the $V_i$'s and note that $|W_1|\geq |V_1|\geq N/r-\eta N\geq N/2r$. Therefore, if $G_R[V_1]$ is not $\mathcal{T}(n,D)$-universal, then Theorem~\ref{haxxel} implies that $|V_i|\leq |V_1|\leq n + 3D\eta N$ for all $i\in[r]$. Set $V_0=V(G)\setminus(V_1\cup \dots\cup V_r)$ and choose $\eta$ small enough so that
\[|V_0|\leq \dfrac{\alpha n}{2r} \quad \text{and} \quad |V_i|\leq \left(1+\dfrac{\alpha}{r}\right)n\]
for each $i\in[r]$. To finish the proof we only need to show that $|V_i|\geq (1-\alpha) n$ for each $i\in [r]$. We suppose without loss of generality that $|V_r|<(1-\alpha)n$. Then there exists $j\in[r-1]$ such that
\[|V_j|\geq \dfrac{N-|V_r|-|V_0|}{r-1}> \dfrac{1}{r-1}\left((1-\delta)rn - (1-\alpha)n-\dfrac{\alpha n}{2r}\right)\geq \left(1+\dfrac{\alpha}{r}\right)n,\]
which is a contradiction and thus $||V_i|-n|\leq \alpha n$ for all $i\in[r]$.
\end{proof}
Now we push the stability even further. It will be convenient to relate expansion properties of the red graphs on each part based solely on the red and blue degrees inside this part. We prove that if a set induces a graph with high minimum red degree and roughly the expected codegree, then it satisfies property~\eqref{strong} of expansion.
\begin{lemma}
\label{incexc}
  For every $C,\gamma>0$  there exists $\gamma'>0$ such that the following holds for $p\gg \log N/N$. Let $G$ be an $N$-vertex graph such that for all $u,v\in V(G)$ we have $d(u)\geq \gamma pN$ and $|N(u)\cap N(v)|\leq 2p^2N\log N$. Then for every $X\subseteq V(G)$, with $1\le |X|\leq C/p$, we have $|N(X)|\geq {\gamma' pN}|X|/\log N$.
\end{lemma}
\begin{proof}
For $X\subseteq V(G)$ with $1\le |X|\leq C/p$, take a subset $X'\subseteq X$ with $1\le |X'|\leq \gamma /(4p\log N)$. By inclusion-exclusion, $|N(X)|$ is at least
\begin{equation*}
\begin{split}
 \sum_{u\in X'}|N(u)| -
  \sum_{v\not=w}|N(v)\cap N(w)| - |X| 
  &\geq \gamma pN|X'| -|X'|^ 2\cdot(2p^2N\log N) - |X| \\
  &\geq  \gamma pN|X'| -\dfrac{\gamma pN}{2}|X'| - |X|\\
  &\geq  \Omega \left(\dfrac{pN}{\log N}\right)|X|,
\end{split}
\end{equation*}
 where in the last inequality we used that $pN\gg \log N$.
\end{proof}

\begin{definition}
Let $\eps>0$ and let $r,D\geq 2$ be integers. For a blue-red colored $N$-vertex graph $G$, we say that a partition $V(G)=V_0\cup V_1\cup\dots\cup V_r$ is $\varepsilon$-good if for every $i\in [r]$
\begin{enumerate}[a)]
  \item\label{good:a} $|V_i|\geq (1-1/2D)N/r$,
  \item\label{good:b} $d_R(v,V_i)\geq pN/32r$ for every $v\in V_i$, and 
  \item\label{good:c} $d_B(v,V_i)\leq \varepsilon pN$  for every $v\in V_i$.
\end{enumerate}
\end{definition}
We will prove now that for any $\varepsilon$-good partition of $V(G(N,p))$ we have that $e_R(V_i,V_j)=0$ for all $1\leq i<j\leq r$.  First, we prove that $G_R[V_i]$ is an expander for each $i\in [r]$. Thus, by Haxell's theorem (Theorem~\ref{haxxel}), we can embed any tree of size $(1-o(1))n$ into any of the $V_i$'s. Suppose there is a red edge between $V_i$ and $V_j$.  We may split any given tree $T\in \mathcal{T}(n,D)$ in two trees $T_1$ and $T_2$, connected by an edge and both having at most $(1-1/D)n$ vertices. Then, we can embed $T_1$ into $V_i$ and $T_2$ into $V_j$, and complete the embedding of $T$ using the red edge between $V_i$ and $V_j$. 

Using this fact we can prove that $G[V_i]$ has even stronger expansion properties. That is, for each $i\in [r]$ we may show that every pair of large disjoint subsets of $V_i$ always have at least one red edge in between. Indeed, if  for some $i\in [r]$ there exist a pair of disjoint sets $X,Y\subseteq V_i$ each of size $\Theta(N/\log^4 N)$ and no red edges in between, then, with high probability, $X$ and $Y$ and the remaining $V_j$'s would span a canonical blue-copy of $K_{r+1}$. Combining  this information with results of Section~\ref{sec:expanders}, we show that $G_R[V_i]$ is $\mathcal{T}(|V_i|-C/p,D)$-universal for every $i\in [r]$.

\begin{prop}
  \label{umbreon}
  For integers $r,D\geq 2$ there exist $C,C',\delta,\varepsilon>0$ such that if $N\geq (1-\delta)rn$ and $p\geq C'N^{-2/r+2}$, then $G=G(N,p)$ has, with high probability, the following property. For every blue-red colouring of $E(G)$ that admits an $\varepsilon$-good partition $V(G)=V_0 \cup V_1\cup \dots \cup V_r$, at least one of the following holds:
  \begin{enumerate}[a)]
      \item\label{umbreon:1} $G$ contains a blue copy of $K_{r+1}$.
      \item\label{umbreon:2} $G$ contains a red copy $T\in \mathcal{T}(n,D)$.
      \item\label{umbreon:3} For every $1\leq i<j\leq r$ we have $e_R(V_i,V_j)=0$. Moreover, for each $i\in[r]$ the graph $G_R[V_i]$ is $\mathcal{T}(|V_i|-C/p,D)$-universal.
  \end{enumerate}
\end{prop}
\begin{proof}Assume that neither \eqref{umbreon:1} nor \eqref{umbreon:2} hold.  For $\alpha=1/32D$, we take $C$ from Lemma~\ref{p2} so that, with high probability, $G$ is a weak $(C/p, \alpha N/4r)$-expander, and set $\varepsilon= \alpha/(6CD)$. Moreover, there exists a constant $C'$ such that if $p\geq C'N^{-1/2}$, then, with high probability, every pair of vertices in $G$ has at most $2p^2N \log N$ common neighbours. Finally, because of the first property of the $\varepsilon$-good partition, we deduce that $N\leq 2r|V_i|$. Our first goal is to prove that each $V_i$ satisfies the hypothesis of Theorem~\ref{haxxel} in order to show that $G_R[V_i]$ is $\mathcal T((1-1/D)n,D)$-universal. For $i\in[r]$, we apply Lemma~\ref{incexc} to $G_R[V_i]$, with parameters $\gamma=1/32r$ and $C$, so that for every $X\subseteq V_i$, with $1\le |X|\leq C/p$, we have 
  \begin{equation}
    \label{jynx}|N_R(X)\cap V_i|=\Omega\left(\dfrac{pN}{\log N}\right)|X|\geq D|X|+1.\end{equation}
For $X\subseteq V_i$, with $C/p\leq |X|\leq 2C/p$, since $G$ is a weak $(C/p,\alpha N/4r)$-expander we have
\begin{equation}
  \label{alakazam}
    |N_R(X)\cap V_i| \geq |V_i|-\dfrac{\alpha N}{4r} -\eps pN|X| - |X|
                    \geq \left(1-\alpha\right)|V_i|+D|X|+1.
    \end{equation}
Since $\alpha \leq 1/D$, then $(1-\alpha)|V_i|\geq (1-1/D)n$,
and thus we may use Theorem~\ref{haxxel} on each $G_R[V_i]$ in order to find trees of size $(1-1/D)n$ and maximum degree at most $D$.

Given a tree $T\in \mathcal{T}(n,D)$, there exists a cut edge $u_1u_2\in E(T)$ which splits $T$ into two trees $T_1$ and $T_2$, both with at least $n/D$ vertices and, consequently, at most $(1-1/D)n$ vertices (see~\cite[Lemma~2.5]{BPS3}). Suppose that exists a red edge $v_1v_2$ between two different parts, say $v_1\in V_1$ and $v_2\in V_2$.  By Theorem~\ref{haxxel}, we may find an embedding of $T_i$ in $G_R[V_i]$ that maps $u_i$ to $v_i$, for $i\in\{1,2\}$, and thus, together with the red edge $v_1v_2$, yield an embedding of $T$. Therefore, there are no red edges between different parts. Now we move to prove the second part of~\eqref{umbreon:3}.

Set $d=D\log^4n/20$. We will show now that $G_R[V_i]$ is an $(|V_i|/2d,|V_i|/2d,d)$-expander for each $i\in[r]$. Indeed, given $X\subseteq V_i$, with $1\le |X|\leq C/p$, by~\eqref{jynx} we get  $|N_R(X)\cap V_i|\geq d|X|+1$. For $C/p\leq |X| \leq |V_i|/2d$, by \eqref{alakazam} we have that
\[|N_R(X)\cap V_i|\geq (1-\alpha)|V_i|-|X|\geq d|X|+1,\]
 since $\alpha <1/2$. To show the second expansion property, suppose that there exists a pair of disjoint sets $X,Y\subseteq V_i$, with $|X|=|Y|=|V_i|/2d$, such that $e_R(X,Y)=0$. By Lemma~\ref{Janson}, with high probability there is a copy of $K_{r+1}$ with one vertex in each of the sets $X,Y$ and the $V_j$'s with $j\neq i$ (we can apply Janson's inequality since $|V_i|/2d = \Omega (N/\log^4N)$). This is a contradiction and therefore $G_R[V_i]$ is an $(|V_i|/2d,|V_i|2d,d)$-expander. Now, Theorem~\ref{mont} implies that $G_R[V_i]$ contains all spanning trees with maximum degree bounded by $D$ and at most $|V_i|/d$ leaves.

For trees with at least $|V_i|/d$ leaves, we know that $G_R[V_i]$ is a weak $(|V_i|/2d,|V_i|/2d)$-expander, and so we only need to show that it is also a weak $(C/p,|V_i|/32D)$-expander. But this is already guaranteed by~\eqref{alakazam} since $\alpha\leq 1/32D$. Now, Theorem~\ref{manyleaves} implies that  $G_R[V_i]$ is $\mathcal{T}(|V_i|-C/p,D)$-universal.
\end{proof}

Now we are ready to prove Theorem \ref{finalstab}.
\begin{proof}[Proof of Theorem~\ref{finalstab}]
We apply Proposition~\ref{umbreon}, with parameters $r$ and $D$, to get $\delta_1, \varepsilon, C, C'_1$, and let $\alpha\leq 1/6D$ be sufficiently small. Without loss of generalisation, we assume that $0<\eps\le \alpha/r$ and apply Proposition~\ref{jolteon}, with parameters $\varepsilon^2/4$ and $\alpha$, to get $C'_2$ and $\delta_2$. Let $C'_3$ be given by Lemma~\ref{p3} and set $C'_4= 10^5r^2$. Finally, we set $\delta = \min \{\delta_1,\delta_2\}$ and $C'=\max\{C'_1,C'_2,C'_3,C'_4\}$, and consider $N\geq(1-\delta)rn$ and $p\geq C'N^{-2/(r+2)}$.

By Proposition~\ref{jolteon}, with high probability, if  $K_{r+1}\nsubseteq G_B$ and if $G_R$ is not $\mathcal{T}(n,D)$-universal, then there exists a partition $V(G)=V_0\cup V_1\cup\dots \cup V_r$ such that $|V_0|\leq \alpha n$, and for each $i\in [r]$ we have $||V_i|-n|\leq \alpha n$ and $e_B(V_i)\leq \varepsilon^2pN^2/4$ . We want to define a new partition by removing from each $V_i$ a set of ``bad" vertices. First,  for $i\in [r]$ let $B_i$ be the set of those vertices $v\in V_i$ having at least $\eps pN$ blue neighbours in $V_i$ and set $B=B_1\cup\dots\cup B_r$. Secondly, let $B'$ be the set of those vertices $v\in V(G)$ such that $d(v,V_i\setminus B)\le pN/16r$ for some $i\in [r$].

Let $V(G)=W_0\cup W_1\dots\cup W_r$ be the partition defined by $W_i=V_i\setminus (B\cup B')$ for $i\in [r]$ and $W_0=V(G)\setminus( W_1\cup\dots \cup W_r)$. We will show that this partition is $\varepsilon$-good. Since $e_B(V_i)\leq \varepsilon^2pN^2/4$, a double counting argument shows that $|B\cap V_i|\leq \varepsilon N/2$ and thus $|V_i\setminus B|\ge |V_i|-\eps N/2\ge (1-2\alpha)N/r$ as $\eps\le \alpha /r$.
By Lemma~\ref{p2}, there are at most $128r/p$ vertices of $G$ with less than $pN/16r$ neighbours in $V_i\setminus B$. Then we have 
\[|W_i| \geq (1-2\alpha)\dfrac{N}{r} - \dfrac{128r^2}{p} 
\geq (1-3\alpha)\dfrac{N}{r}\geq \left(1-\dfrac{1}{2D}\right)\dfrac{N}{r}.\]
By definition of $W_i$, each vertex $v\in W_i$ satisfies $d_B(v,W_i)\le \eps pN$. On the other hand, for $v\in W_i$ we have
\[d_R(u,W_i)\geq \dfrac{pN}{16r}-{\varepsilon pN} - \dfrac{128r^2}{p}\geq \dfrac{pN}{32r},\]
where we used that $\varepsilon\leq 1/20r$ and  $pN\geq C_4/p$. To finish the proof, take an  $\varepsilon$-good partition $V(G)=U_0\cup U_1\cup\dots\cup U_r$ such that $W_i\subseteq U_i$ for $i\in[r]$ and that minimizes $|U_0|$. We will prove that if $U_0\nsubseteq B'$, then this partition would not be maximal. By contradiction, suppose there exists $u\in U_0\setminus B'$. If $d_B(u,U_i)\geq \varepsilon pN$ for all $i\in [r]$, then by Lemma~\ref{p3} we can find a blue copy of $K_{r+1}$ containing $u$, which is not possible. Then there must exist some $i\in [r]$ such that $d_R(u,U_i)\geq pN/32r$, in which case we update $U_i:=U_i\cup \{u\}$. We claim that $V(G)=U_0\cup U_1\cup\dots \cup U_r$ is still $\varepsilon$-good. Since the blue degree of each vertex in $U_i\setminus\{u\}$ grows in at most 1, it follows that the new partition is $2\eps$-good. This fact and Proposition~\ref{umbreon} imply that $e_R(U_i,U_j)=0$ for every $1\leq i<j\leq r$. Finally, we may use Lemma~\ref{p3} as before to show that the maximum blue degree inside each part is at most $\varepsilon pN$, which makes this partition $\varepsilon$-good. This contradicts the maximality of the initial partition and thus $U_0\subseteq B'$. In particular, we have $|U_0|\le |B'|\le 128r/p$. Note that if $|U_i|>(n+C/p)$ for some $i\in [r]$, then, by Proposition~\ref{umbreon}, $G_R[U_i]$ contains all trees with maximum degree at most $D$ and                                                                              $|U_i|-C/p\ge n$ edges, which is a contradiction. This finishes the proof.

\end{proof}

\bibliographystyle{acm}
\bibliography{bib}

\end{document}